\documentclass[11pt]{amsart}

\usepackage{amsmath,amssymb,amsthm}
\usepackage{amsmath,amssymb,amsthm,amscd}
\usepackage[frame,cmtip,arrow,matrix,line,graph,curve]{xy}
\usepackage{graphpap, color}
\usepackage[mathscr]{eucal}
\usepackage{color}

\def\dfrac{\displaystyle\frac}
\def\dsum{\displaystyle\sum}

\newtheorem{prop}{Proposition}
\newtheorem{theo}[prop]{Theorem}
\newtheorem{lemm}[prop]{Lemma}
\newtheorem{coro}[prop]{Corollary}

\newtheorem{defi}[prop]{Definition}

\newcommand{\al}{\alpha}
\newcommand{\abc}[1]{\left( #1 \right)}%
\newcommand{\abz}[1]{\left[ #1 \right]}%[]
\renewcommand{\leq}{\leqslant}
\renewcommand{\geq}{\geqslant}

\newcommand{\p}{\partial}

\newcommand{\la}{\kappa}

\numberwithin{equation}{section}

\begin{document}

\title{On the curvature estimates for Hessian equations}

\author{Changyu Ren}
\address{School of Mathematical Science\\
Jilin University\\ Changchun\\ China}
\email{rency@jlu.edu.cn}
\author{Zhizhang Wang}
\address{School of Mathematical Science\\ Fudan University \\ Shanghai, China}
\email{zzwang@fudan.edu.cn}
\thanks{Research of the  last author is supported  by an NSFC Grant No.11301087}
\begin{abstract}
The curvature estimates of $k$ curvature equations for general right hand side is a longstanding problem. In this paper, we totally solve the $n-1$ case and we also discuss some applications for our estimate. 
\end{abstract}
\maketitle

\section{introduction}

In this paper, we continue to  study the longstanding problem of global $C^2$ estimates for curvature equation in general type, 
\begin{eqnarray}\label{1.1}
\sigma_k(\kappa(X))=f(X,\nu(X)), \ \ \forall  X\in M,
\end{eqnarray}
where $\sigma_k$ is the $k$th elementary symmetric function, $\nu(X), \kappa(X)$ are the outer-normal and principal curvatures of hypersurface $M\subset \mathbb R^{n+1}$ at the position vector $X$ respectively. 

Equation \eqref{1.1} is the general form of some important type equations. For the cases $k=1,2$ and $n$, they are the mean curvature, scalar curvature and Gauss curvature type equation. We will mainly  discuss  the case of $k=n-1$ in this paper.

Now, let's give a brief review of some history related these equations. A lot of geometric problems fall into equation (\ref{1.1}) with special form of $f$. The  famous  Minkowski problem, namely, prescribed Gauss-Kronecker curvature on the outer normal, has been widely discussed in \cite{N, P1, P3, CY}. Alexandrov also posed the problem of prescribing general Weingarten curvature on outer normals, seeing  \cite{A2, gg}. The prescribing curvature measures problem in convex geometry also has been extensively studied in \cite{A1, P1, GLM, GLL}. In \cite{BK, TW, CNS5}, the prescribing mean curvature problem and Weingarten curvature problem also have been considered and obtained fruitful results. 

In many case, the main difficulty of the equation \eqref{1.1} is trying to obtain $C^2$ estimates. Hence, let's review some known results. For $k=1$, equation (\ref{1.1}) is quasilinear, $C^2$ estimate follows from the classical theory of quasilinear PDE.  The equation is of Monge-Amp\`ere type if $k=n$. $C^2$ estimate in this case for general $f(X, \nu)$ is due to Caffarelli-Nirenberg-Spruck \cite{CNS1}. When $f$ is independent of normal vector $\nu$, $C^2$ estimate has been proved by Caffralli-Nirenberg-Spruck \cite{CNS5}. If $f$ in (\ref{1.1}) depends only on $\nu$, $C^2$ estimate was proved in \cite{gg}. Ivochkina \cite{I1, I} considered the Dirichlet problem of equation (\ref{1.1}) on domains in $\mathbb R^n$, $C^2$ estimate was proved there under some extra conditions on the dependence of $f$ on $\nu$. $C^2$ estimate was also proved for equation of prescribing curvature measures problem in \cite{GLM, GLL}, where $f(X,\nu)= \langle X,\nu \rangle \tilde f(X)$. For $k=2$ and convex case, the $C^2$ estimate have been obtained in \cite{GRW}. Recently, the scalar curvature case is generalized and simplified in \cite{SX}.  For general equation \eqref{1.1}, the desired $C^2$ estimate should be in the Grading cone $\Gamma_k$. Following  \cite{CNS3}, the Garding's cone is defined by, 

\begin{defi}\label{k-convex} For a domain $\Omega\subset \mathbb R^n$, a function $v\in C^2(\Omega)$ is called $k$-convex if the eigenvalues $\kappa (x)=(\kappa_1(x), \cdots, \kappa_n(x))$ of the hessian $\nabla^2 v(x)$ is in $\Gamma_k$ for all $x\in \Omega$, where $\Gamma_k$ is the Garding's cone
\[\Gamma_k=\{\kappa \in \mathbb R^n \ | \quad \sigma_m(\kappa)>0, \quad  m=1,\cdots,k\}.\]

A $C^2$ regular hypersurface $M\subset \mathbb R^{n+1}$ is
$k$-convex if $\kappa(X)\in \Gamma_k$
for all $X\in M$.
 \end{defi}
In the present paper, for $n-1$ Hessian equation, we can obtain the $C^2$ estimate in $\Gamma_{n-1}$. Namely, totally solve the $C^2$ estimate for $n-1$ Hessian equation. In fact, the main result of this paper is, 

\begin{theo}\label{theo2}
Suppose $M\subset \mathbb R^{n+1}$ is a closed ${n-1}$-convex hypersurface satisfying curvature equation (\ref{1.1}) with $k=n-1$ for some positive function $f(X, \nu)\in C^{2}(\Gamma)$, where $\Gamma$ is an open neighborhood of unit normal bundle of $M$ in $\mathbb R^{n+1} \times \mathbb S^n$, then there is a constant $C$ depending only on $n, k$, $\|M\|_{C^1}$, $\inf f$ and $\|f\|_{C^2}$, such that
 \begin{equation}\label{Mc2}
 \max_{X\in M, i=1,\cdots, n} \kappa_i(X) \le C.\end{equation}
\end{theo}
We use two steps to prove the above estimate.  The first key step is to obtain a better inequality which we have got in section 2. This is more explicit estimate than  the inequalities obtained in \cite{GRW}. Then using the test function discovered in \cite{GRW}, we obtain the global $C^2$ estimate.

We also have the similar estimate for Direchlet problem in $\mathbb{R}^n$.
\begin{coro} \label{Coro}
For the Direchlet problem of $\sigma_{n-1}$ equation defined in some bounded domain $\Omega\subset \mathbb{R}^n$, it is, 
\begin{eqnarray}\label{1.3}
\left\{\begin{matrix}\sigma_{n-1}[D^2u]&=&f(x,u,Du), & \text{ in } & \Omega\\  u&=&\varphi,   &\text{ on } &\partial \Omega\end{matrix}\right.
\end{eqnarray}
The global $C^2$ estimates can be obtained. It means that, we have some constants $C$ depending on $f$ and $\nabla u, u$ and the domain $\Omega$, such that,
$$\|u\|_{C^2(\bar{\Omega})}\leq C+\max_{\p\Omega}|\nabla^2 u|.$$
\end{coro}
More reference about these type of estimates can be found in \cite{CW}, \cite{LRW} and therein.

Now, let's exhibit some applications of our estimate. The first application is that we can obtain the corresponding existence result for $n-1$-convex solutions of the prescribed $n-1$ curvature equation (\ref{1.1}). For the sake of the $C^0,C^1$ estimates, we need further barrier conditions on the prescribed function $f$ as considered in \cite{BK, TW, CNS5}. We denote $\rho(X)=|X|$.

\medskip

We assume that

\noindent {\it Condition} (1). There are two positive constant $r_1<1<r_2$ such that
\begin{equation}\label{4.1}
\left\{
\begin{matrix}
f(X,\frac{X}{|X|}) &\geq&  \dfrac{\sigma_k(1,\cdots, 1)}{r^k_1},\ \ \text{ for } |X|=r_1,\\
f(X,\frac{X}{|X|}) &\leq&  \dfrac{\sigma_k(1,\cdots, 1)}{r_2^k}, \ \ \text{  for  } |X|=r_2 .
\end{matrix}\right.
\end{equation}

\noindent {\it Condition} (2). For any fixed unit vector $\nu$,
\begin{eqnarray}\label{4.2}
\frac{\p }{\p \rho}(\rho^kf(X,\nu))\leq 0,\ \   \text{ where } |X|=\rho.
\end{eqnarray}

Using the above two condition, we have the following existence theorem. 
\begin{theo}\label{exist} Suppose $k=n-1$ and suppose positive function $f\in C^2(\bar B_{r_2}\setminus B_{r_1}\times \mathbb S^n)$ satisfies conditions (\ref{4.1}) and (\ref{4.2}), then equation (\ref{1.1}) has a unique $C^{3,\alpha}$ starshaped solution $M$ in $\{r_1\le |X|\le r_2\}$.
\end{theo}

\bigskip
We also can apply our estimate to the prescribed curvature problem for spacelike graph hypersurface in Minkowski space. We assume the graph can be written by function $u$ which means  that $(x,u(x)), x\in \mathbb{R}^n$ is its position vector. Still, we suppose $\kappa_1,\cdots,\kappa_n$ be the principal curvature of these hypersurface. The principal curvature can be written by the derivative of the function $u$ which will be more clear in section 4. We have the following theorem.
\begin{theo}\label{theo5}
 Let $\Omega$ be some bounded domain in $\mathbb{R}^n$ with smooth boundary and $f\in C^2(\bar{\Omega}\times \mathbb{R}\times \mathbb{R}^n)$ is a positive function with $f_u\geq 0$. Let  $\varphi\in C^4(\bar{\Omega})$ be space like. Consider  the following Dirichlet problem, 
\begin{eqnarray}\label{1.66}
\left\{\begin{matrix}\sigma_{n-1}(\kappa_1,\cdots,\kappa_n)&=&f(x,u,Du), & \text{ in } & \Omega\\  u&=&\varphi,   &\text{ on } &\partial \Omega\end{matrix}\right..
\end{eqnarray}
If the above problem have some sub soultion, then it has a unique space like solution $u$ in $\Gamma_{n-1}$ belonging to $C^{3,\alpha}(\bar{\Omega})$ for any $\alpha\in (0,1)$.
\end{theo}
The prescribed curvature problem for spacelike graph hypersurface in Minkowski space is proposed by Bayard \cite{Ba0,Ba}. The scalar curvature case has been totally solved by Urban \cite{U}.  The above theorem solves $k=n-1$ case.  For the rest case $2<k<n-1$, it is still open. The difference with the problem in Euclidean space is that the curvature term has opposite sign. Hence, even for function $f$ does not depend on gradient term, these problem can not be successful solved as in Euclidean space, comparing \cite{CNSV}.  Hypersurfaces of prescribed curvature problem in Lorentzian manifolds also have been extensively studied by Bartnik-Simon \cite{BS}, Delano\"e \cite{De}, Gerhardt \cite{Ger1,Ger2} and Schn\"urer \cite{Sch}.

In this paper, we use standard notation.  We let $\kappa(A)$ be eigenvalues of the matrix $A=(a_{ij})$.   For equation $$F(A)=F(\kappa (A)),$$ we define
$$F^{pq}=\frac{\p F}{\p a_{pq}}, \text{ and  } F^{pq,rs}=\frac{\p^2 F}{\p a_{pq}\p a_{rs}}.$$ For a local orthonormal frame, if $A$ is diagonal at a point, then at this point, $$F^{pp}=\frac{\p f}{\p \kappa_p}=f_p, \text{ and }  F^{pp,qq}=\frac{\p^2 f}{\p \kappa_p\p \kappa_q}=f_{pq}.$$ The following facts regarding $\sigma_k$ will be used throughout this paper.

\noindent (i) $\sigma^{pp,pp}_k=0$ and $\sigma^{pp,qq}_k(\kappa)=\sigma_{k-2}(\kappa|pq)$;
 
\noindent (ii)  $\sigma^{pq,rs}_kh_{pql}h_{rsl}=\sigma_k^{pp,qq}h_{pql}^2-\sigma^{pp,qq}_kh_{ppl}h_{qql}$.

Here, the notation $\sigma_l(\kappa|ab\cdots)$ means $l$ symmetric function exclude the indices $a,b,\cdots$. 
Now, we give the following two Lemmas, which will be needed in our proof.
\par
\begin{lemm} \label{Guan}
Set $k>l$. For $\al=\dfrac{1}{k-l}$, we have,
\begin{eqnarray}\label{1.7}
&&-\frac{\sigma_k^{pp,qq}}{\sigma_k}u_{pph}u_{qqh}+\dfrac{\sigma_l^{pp,qq}}{\sigma_l}u_{pph}u_{qqh}\\
&\geq& \abc{\dfrac{(\sigma_k)_h}{\sigma_k}-\dfrac{(\sigma_l)_h}{\sigma_l}}
\abc{(\al-1)\dfrac{(\sigma_k)_h}{\sigma_k}-(\al+1)\dfrac{(\sigma_l)_h}{\sigma_l}}.\nonumber
\end{eqnarray}
further more, for sufficiently small $\delta>0$, we have,
\begin{eqnarray}\label{1.8}
&&-\sigma_k^{pp,qq}u_{pph}u_{qqh} +(1-\al+\dfrac{\al}{\delta})\dfrac{(\sigma_k)_h^2}{\sigma_k}\\
&\geq& \sigma_k(\al+1-\delta\al)
\abz{\dfrac{(\sigma_l)_h}{\sigma_l}}^2
-\dfrac{\sigma_k}{\sigma_l}\sigma_l^{pp,qq}u_{pph}u_{qqh}.\nonumber
\end{eqnarray}
\end{lemm}
\par
The another one is,
\begin{lemm} \label{lemm D}
Denote $Sym(n)$ the set of all $n\times n$ symmetric matrices. Let $F$ be a $C^2$ symmetric function defined in some open subset $\Psi \subset Sym(n)$. At any diagonal matrix $A\in \Psi$ with distinct
eigenvalues, let $\ddot{F}(B,B)$ be the second derivative of $C^2$ symmetric function $F$
in direction $B \in Sym(n)$, then
\begin{eqnarray}
\label{1.9} \ddot{F}(B,B) =  \sum_{j,k=1}^n {\ddot{f}}^{jk}
B_{jj}B_{kk} + 2 \sum_{j < k} \frac{\dot{f}^j -
\dot{f}^k}{{\kappa}_j - {\kappa}_k} B_{jk}^2.
\end{eqnarray}
\end{lemm}
The proof of the first Lemma can be found in \cite{GLL} and \cite{GRW}. The second Lemma can be found in \cite{Ball} and \cite{CNS3}.

The organization of the paper is as follow. We give the key inequality in section 2. Theorem \ref{theo2} is proved in section 3. in section 4, we obtain some applications.

\section{An inequality}
\par
In this section, we will prove the following Proposition. It is a explicit inequality.  We consider the $\sigma_{n-1}$ equation in $n$ dimensional space.
\begin{prop} \label{lea}
For any  index  $i$ and $\varepsilon$, if $\kappa_i\geq \delta \kappa_1$, then we have,
\begin{equation}\label{e1.1}
\la_i[K(\sigma_{n-1})_i^2-\sigma_{n-1}^{pp,qq}u_{ppi}u_{qqi}]-\sigma^{ii}_{n-1}u_{iii}^2+(1+\varepsilon)\sum_{j\neq i}\sigma_{n-1}^{jj}u_{jji}^2\geq 0.
\end{equation}
for sufficient large $K$ depending on $\delta$ and $\varepsilon$.
\end{prop}
\begin{proof}
A directly calculation shows,
\begin{eqnarray}\label{e1.2}
&&\la_i[K(\sigma_{n-1})_i^2-\sigma_{n-1}^{pp,qq}u_{ppi}u_{qqi}]-\sigma_{n-1}^{ii}u_{iii}^2+(1+\varepsilon)\sum_{j\neq i}\sigma_{n-1}^{jj}u_{jji}^2\\
&=&\la_iK[\sum_{j\neq i}\sigma_{n-1}^{jj}u_{jji}]^2+2\la_iu_{iii}[\sum_{j\neq i}(K\sigma_{n-1}^{ii}\sigma_{n-1}^{jj}-\sigma_{n-1}^{ii,jj})u_{jji}]\nonumber\\&&+(\la_iK(\sigma_{n-1}^{ii})^2-\sigma_{n-1}^{ii})u_{iii}^2+(1+\varepsilon)\sum_{j\neq i}\sigma^{jj}_{n-1}u_{jji}^2-\la_i\sum_{p\neq i; q\neq i}\sigma_{n-1}^{pp,qq}u_{ppi}u_{qqi}\nonumber\\
&\geq&\la_iK[\sum_{j\neq i}\sigma_{n-1}^{jj}u_{jji}]^2-\frac{\la_i^2[\sum_{j\neq i}(K\sigma_{n-1}^{ii}\sigma_{n-1}^{jj}-\sigma_{n-1}^{ii,jj})u_{jji}]^2}{\la_iK(\sigma_{n-1}^{ii})^2-\sigma_{n-1}^{ii}}\nonumber\\&&+(1+\varepsilon)\sum_{j\neq i}\sigma^{jj}_{n-1}u_{jji}^2-\la_i\sum_{p\neq i; q\neq i}\sigma_{n-1}^{pp,qq}u_{ppi}u_{qqi}\nonumber\\
&=&\sum_{j\neq i}[\la_iK(\sigma_{n-1}^{jj})^2-\frac{\la_i^2(K\sigma_{n-1}^{ii}\sigma_{n-1}^{jj}-\sigma_{n-1}^{ii,jj})^2}{\la_iK(\sigma_{n-1}^{ii})^2-\sigma_{n-1}^{ii}}+(1+\varepsilon)\sigma_{n-1}^{jj}]u_{jji}^2\nonumber\\
&&+\sum_{p,q\neq i;p\neq q}[\la_iK\sigma_{n-1}^{pp}\sigma_{n-1}^{qq}-\frac{\la_i^2(K\sigma_{n-1}^{ii}\sigma_{n-1}^{pp}-\sigma_{n-1}^{ii,pp})(K\sigma_{n-1}^{ii}\sigma_{n-1}^{qq}-\sigma_{n-1}^{ii,qq})}{\la_iK(\sigma_{n-1}^{ii})^2-\sigma_{n-1}^{ii}}\nonumber\\&&-\la_i\sigma_{n-1}^{pp,qq}]u_{ppi}u_{qqi}\nonumber,
\end{eqnarray}
where, in the second inequality, we have used,
\begin{eqnarray}
&&\frac{\la_i^2[\sum_{j\neq i}(K\sigma_{n-1}^{ii}\sigma_{n-1}^{jj}-\sigma_{n-1}^{ii,jj})u_{jji}]^2}{\la_iK(\sigma_{n-1}^{ii})^2-\sigma_{n-1}^{ii}}+2\la_iu_{iii}[\sum_{j\neq i}(K\sigma_{n-1}^{ii}\sigma_{n-1}^{jj}-\sigma_{n-1}^{ii,jj})u_{jji}]\nonumber\\ &&+(\la_iK(\sigma_{n-1}^{ii})^2-\sigma_{n-1}^{ii})u_{iii}^2\geq 0. \nonumber
\end{eqnarray}
Note that  we have, 
$$K\kappa_i\sigma_{n-1}^{ii}-1\geq K\delta \kappa_1\sigma_{n-1}^{11}-1\geq 0,$$ for sufficient large $K$. Hence, we can omit the denominator in \eqref{e1.2}. Then,  we get,
\begin{eqnarray}\label{e1.3}
&&(\la_iK(\sigma_{n-1}^{ii})^2-\sigma_{n-1}^{ii})[\la_i[K(\sigma_{n-1})_i^2-\sigma_{n-1}^{pp,qq}u_{ppi}u_{qqi}]-\sigma_{n-1}^{ii}u_{iii}^2\\&&+(1+\varepsilon)\sum_{j\neq i}\sigma_{n-1}^{jj}u_{jji}^2]\nonumber\\
&\geq&\sum_{j\neq i}[\la_iK\sigma_{n-1}^{ii}\sigma_{n-1}^{jj}(-\sigma_{n-1}^{jj}+2\kappa_i\sigma_{n-1}^{ii,jj}+(1+\varepsilon)\sigma_{n-1}^{ii})-\la_i^2(\sigma_{n-1}^{ii,jj})^2\nonumber\\&&-(1+\varepsilon)\sigma_{n-1}^{ii}\sigma_{n-1}^{jj}]u_{jji}^2\nonumber\\
&&+\sum_{p,q\neq i;p\neq q}[\la_iK\sigma_{n-1}^{ii}(\la_i(\sigma_{n-1}^{pp}\sigma_{n-1}^{ii,qq}+\sigma_{n-1}^{qq}\sigma_{n-1}^{ii,pp}-\sigma_{n-1}^{ii}\sigma_{n-1}^{pp,qq})-\sigma_{n-1}^{pp}\sigma_{n-1}^{qq})\nonumber\\&&-\la_i^2\sigma_{n-1}^{ii,pp}\sigma_{n-1}^{ii,qq}+\la_i\sigma_{n-1}^{ii}\sigma_{n-1}^{pp,qq}]u_{ppi}u_{qqi}\nonumber.
\end{eqnarray}
We have several identities. At first, we have,
$$-\sigma_{n-1}^{jj}+2\la_i\sigma_{n-1}^{ii,jj}+\sigma_{n-1}^{ii}=(\la_i+\la_j)\sigma_{n-3}(\la|ij).$$
Hence, we get,
\begin{eqnarray}\label{e1.4}
&&\sigma_{n-1}^{jj}(\la_i+\la_j)\sigma_{n-3}(\la|ij)\\
&=&(\la_i\sigma_{n-2}(\la|j)+\sigma_{n-1}-\sigma_{n-1}(\la|j))\sigma_{n-3}(\la|ij)\nonumber\\
&=&(\la_i(\la_i\sigma_{n-3}(\la|ij)+\sigma_{n-2}(\la|ij))-\la_i\sigma_{n-2}(\la|ij)+\sigma_{n-1})\sigma_{n-3}(\la|ij)\nonumber\\
&=&\la_i^2(\sigma_{n-3}(\la|ij))^2+\sigma_{n-1}\sigma_{n-3}(\la|ij)\nonumber,
\end{eqnarray}
where we have used $\sigma_{n-1}(\kappa|ij)=0$.
We also have,
\begin{eqnarray}\label{e1.5}
&&\la_i(\sigma_{n-1}^{pp}\sigma_{n-1}^{ii,qq}+\sigma_{n-1}^{qq}\sigma_{n-1}^{ii,pp}-\sigma_{n-1}^{ii}\sigma_{n-1}^{pp,qq})-\sigma_{n-1}^{pp}\sigma_{n-1}^{qq}\\
&=&\la_i\sigma_{n-1}^{qq}\sigma_{n-1}^{ii,pp}-\la_i\sigma_{n-1}^{ii}\sigma_{n-1}^{pp,qq}-\sigma_{n-1}^{pp}\sigma_{n-2}(\la|iq)\nonumber\\
&=&\la_i\sigma_{n-2}(\la|pq)\sigma_{n-3}(\la|ip)-\la_i\sigma_{n-2}(\la|ip)\sigma_{n-3}(\la|pq)-\sigma_{n-2}(\la|p)\sigma_{n-2}(\la|iq)\nonumber\\
&=&(\la_i^2\sigma_{n-3}(\la|ipq)+\la_i\sigma_{n-2}(\la|ipq))(\la_q\sigma_{n-4}(\la|ipq)+\sigma_{n-3}(\la|ipq))\nonumber\\
&&-(\la_q\sigma_{n-3}(\la|ipq)+\sigma_{n-2}(\la|ipq))(\la_i^2\sigma_{n-4}(\la|ipq)+\la_i\sigma_{n-3}(\la|ipq))\nonumber\\
&&-\sigma_{n-2}(\la|p)(\la_p\sigma_{n-3}(\la|ipq)+\sigma_{n-2}(\la|ipq))\nonumber\\
&=&\la_i^2(\sigma_{n-3}(\la|ipq))^2-\la_i\la_q(\sigma_{n-3}(\la|ipq))^2-\la_p\sigma_{n-2}(\la|p)\sigma_{n-3}(\la|ipq)\nonumber\\
&=&\la_i^2(\sigma_{n-3}(\la|ipq))^2-\sigma_{n-1}\sigma_{n-3}(\la|ipq).\nonumber
\end{eqnarray}
Here we have used $\sigma_{n-2}(\la|ipq)=0$ and
$$\sigma_{n-1}=\la_p\sigma_{n-2}(\la|p)+\sigma_{n-1}(\la|p)=\la_p\sigma_{n-2}(\la|p)+\la_i\la_q\sigma_{n-3}(\la|ipq).$$
We also have,
\begin{eqnarray}\label{e1.6}
&&\sigma_{n-1}^{ii,pp}\sigma_{n-1}^{ii,qq}\\
&=&(\la_q\sigma_{n-4}(\la|ipq)+\sigma_{n-3}(\la|ipq))(\la_p\sigma_{n-4}(\la|ipq)+\sigma_{n-3}(\la|ipq))\nonumber\\
&=&(\sigma_{n-3}(\la|ipq))^2+[\la_p\la_q\sigma_{n-4}(\la|ipq)+(\la_p+\la_q)\sigma_{n-3}(\la|ipq)]\sigma_{n-4}(\la|ipq)\nonumber\\
&=&(\sigma_{n-3}(\la|ipq))^2+\sigma_{n-2}(\la|i)\sigma_{n-4}(\la|ipq),\nonumber
\end{eqnarray}
where we have used
$$\sigma_{n-2}(\la|i)=\la_p\sigma_{n-3}(\la|ip)+\sigma_{n-2}(\la|ip)=\la_p\la_q\sigma_{n-4}(\la|ipq)+(\la_p+\la_q)\sigma_{n-3}(\la|ipq).$$
We have,
\begin{eqnarray}\label{e1.7}
\sigma_{n-1}^{pp,qq}=\la_i\sigma_{n-4}(\la|ipq)+\sigma_{n-3}(\la|ipq)
\end{eqnarray}

Using the above two identities \eqref{e1.6} and \eqref{e1.7}, we get,
\begin{eqnarray}\label{e1.8}
&&-\la_i^2\sigma_{n-1}^{ii,pp}\sigma_{n-1}^{ii,qq}+\la_i\sigma_{n-1}^{ii}\sigma_{n-1}^{pp,qq}\\
&=&-\la_i^2(\sigma_{n-3}(\la|ipq))^2+\la_i\sigma_{n-1}^{ii}\sigma_{n-3}(\la|ipq).\nonumber
\end{eqnarray}
Using identities \eqref{e1.4}, \eqref{e1.5} and \eqref{e1.8},
\eqref{e1.3} becomes,
\begin{eqnarray*}
&&(\la_iK(\sigma_{n-1}^{ii})^2-\sigma_{n-1}^{ii})[\la_i[K(\sigma_{n-1})_i^2-\sigma_{n-1}^{pp,qq}u_{ppi}u_{qqi}]-\sigma_{n-1}^{ii}u_{iii}^2\\
&&+(1+\varepsilon)\sum_{j\neq i}\sigma_{n-1}^{jj}u_{jji}^2]\nonumber\\
&\geq&\sum_{j\neq i}[\la_iK\sigma_{n-1}^{ii}(\la_i^2(\sigma_{n-3}(\la|ij))^2+\sigma_{n-3}(\la|ij)\sigma_{n-1}+\varepsilon\sigma_{n-1}^{jj}\sigma_{n-1}^{ii})\nonumber\\&&-\la_i^2(\sigma_{n-1}^{ii,jj})^2-(1+\varepsilon)\sigma_{n-1}^{ii}\sigma_{n-1}^{jj}]u_{jji}^2\nonumber\\
&&+\sum_{p,q\neq i,p\neq q}[\la_iK\sigma_{n-1}^{ii}(\la_i^2(\sigma_{n-3}(\la|ipq))^2-\sigma_{n-1}\sigma_{n-3}(\la|ipq))\nonumber\\&&-\la_i^2(\sigma_{n-3}(\la|ipq))^2+\la_i\sigma_{n-1}^{ii}\sigma_{n-3}(\la|ipq)]u_{ppi}u_{qqi}\nonumber\\
&=&\sum_{j\neq i}[(\la_iK\sigma_{n-1}^{ii}-1)\la_i^2(\sigma_{n-3}(\la|ij))^2+\la_iK\sigma_{n-1}^{ii}\sigma_{n-3}(\la|ij)\sigma_{n-1}\nonumber\\&&+(\la_iK\sigma_{n-1}^{ii}\varepsilon-(1+\varepsilon))\sigma_{n-1}^{ii}\sigma_{n-1}^{jj}]u_{jji}^2\nonumber\\
&&+\sum_{p,q\neq i,p\neq q}[(\la_iK\sigma_{n-1}^{ii}-1)\la_i^2(\sigma_{n-3}(\la|ipq))^2\nonumber\\&&-\la_iK\sigma_{n-1}^{ii}(\sigma_{n-1}-\frac{1}{K})\sigma_{n-3}(\la|ipq)]u_{ppi}u_{qqi}\nonumber\\
&\geq&\sum_{j\neq i}[(\la_iK\sigma_{n-1}^{ii}-1)\la_i^2(\sigma_{n-3}(\la|ij))^2+\la_iK\sigma_{n-1}^{ii}\sigma_{n-3}(\la|ij)(\sigma_{n-1}-\frac{1}{K})]u_{jji}^2\nonumber\\
&&+\sum_{p,q\neq i,p\neq q}[(\la_iK\sigma_{n-1}^{ii}-1)\la_i^2(\sigma_{n-3}(\la|ipq))^2\nonumber\\&&-\la_iK\sigma_{n-1}^{ii}(\sigma_{n-1}-\frac{1}{K})\sigma_{n-3}(\la|ipq)]u_{ppi}u_{qqi}\nonumber.
\end{eqnarray*}
Here, the last inequality holds for sufficient large $K$.
Now, we only need to check whether the following two bilinear form are nonnegative. There are
\begin{eqnarray}\label{e1.9}
\sum_{j\neq i}(\sigma_{n-3}(\la|ij))^2u_{jji}^2+\sum_{p,q\neq i,p\neq q}(\sigma_{n-3}(\la|ipq))^2u_{ppi}u_{qqi},
\end{eqnarray}
and,
\begin{eqnarray}\label{e1.10}
\sum_{j\neq i}\sigma_{n-3}(\la|ij)u_{jji}^2-\sum_{p,q\neq i,p\neq q}\sigma_{n-3}(\la|ipq)u_{ppi}u_{qqi}.
\end{eqnarray}
Let's consider the corresponding two matrices.
Denote
$$a_{pq}=\left\{\begin{matrix}\sigma_{n-3}(\la|ip),& p=q\\ -\sigma_{n-3}(\la|ipq),& p\neq q \end{matrix}\right..$$
Now we need a elemental theorem in linear algebra. That is the Schur product theorem for Hadmard product.
\begin{theo}
The Hadmard product of two semipositive definite matrices is semipositive definite.
\end{theo}
Here, the meaning of the Hadmard product is that every entry of the
product of two matrices is the directly product of corresponding entries of two
matrices. For example, if matrices $B=(b_{ij}), C=(c_{ij})$, then the Hadamard product of matrices $B,C$ is the matrix $(b_{ij}c_{ij})$.  Thus, to prove the bilinear forms of \eqref{e1.9} and
\eqref{e1.10} are semi positive forms, we  only need to prove the
matrices $(a_{pq})$ and $(a_{pq}^2)$ are semi positive definite. By
Schur's product theorem, we only need to check that the matrix
$(a_{pq})$ is semi positive definite. It comes from the following
Lemma.
\end{proof}

\begin{lemm}
Suppose $2\leq i_1\leq i_2 \leq \cdots \leq  i_m\leq n$ are  $m$ ordered indices.  Then, $D_m(i_1,\cdots,i_m)$  the $k$-th principal  sub determinant of the matrix $(a_{pq})$ is
\begin{eqnarray}\label{e1.11}
D_m(i_1,\cdots,i_m)&=&\det\left[\begin{matrix}a_{i_1i_1}&a_{i_1i_2}&\cdots&a_{i_1i_m}\\ a_{i_2i_1}&a_{i_2i_2}&\cdots&a_{i_2i_m}\\ \cdots &&  \cdots \\ a_{i_mi_1}& a_{i_mi_2}&\cdots&a_{i_mi_m}\end{matrix}\right]\\
&=&\sigma_{n-2}^{m-1}(\la|1)\sigma_{n-(m+2)}(\la|1i_1\cdots i_m)\nonumber.
\end{eqnarray}
The another needed determinant is, for $k\neq m$,
\begin{eqnarray}\label{e1.12}
&&B_{m-1}(i_1,\cdots,i_m; i_{k})\\
&=&\det\left[\begin{matrix}a_{i_1i_1}&a_{i_1i_2}&\cdots&a_{i_1i_{k-1}}&a_{i_1i_{k+1}}&\cdots&a_{i_1i_m}\\ a_{i_2i_1}&a_{i_2i_2}&\cdots&a_{i_2i_{k-1}}&a_{i_2i_{k+1}}&\cdots&a_{i_2i_m}\\ \cdots &&  \cdots &&&\cdots\\ a_{i_{m-1}i_1}& a_{i_{m-1}i_2}&\cdots&a_{i_{m-1}i_{k-1}}&a_{i_{m-1}i_{k+1}}&\cdots&a_{i_{m-1}i_m}\end{matrix}\right]\nonumber\\
&=&(-1)^{m+k}[\sigma_{n-3}(\la|1i_ki_m)D_{m-2}(i_1\cdots i_{k-1}i_{k+1}\cdots i_{m-1})\nonumber\\&&+\sigma_{n-m}(\la|1i_2\cdots i_m)\sigma_{n-2}^{m-3}(\la|1)\sum_{l\neq k,m}\sigma_{n-3}(\la|1i_1 i_l)]\nonumber.
\end{eqnarray}
Hence, in $\Gamma_{n-1}$ cone, we have, $$D_{m-1}(2\cdots m)=\sigma_{n-2}^{m-2}(\la|1)\sigma_{n-(m+1)}(\la|12\cdots m)>0,$$ which implies the matrix $(a_{pq})$ is a nonnegative definite matrix.
\end{lemm}
\begin{proof}
We prove the above two formulas by induction.

For $m=2$,
$$B_1(i_1i_2;i_1)=a_{i_1i_2}=-\sigma_{n-3}(\la|1i_1i_2).$$ Also, we
have, by \eqref{e1.6},
\begin{eqnarray}
D_2(i_1i_2)&=&\sigma_{n-3}(\la|1i_1)\sigma_{n-3}(\la|1i_2)-\sigma_{n-3}^2(\la|1i_1i_2)\nonumber\\
&=&\sigma_{n-2}(\la|1)\sigma_{n-4}(\la|1i_1i_2).\nonumber
\end{eqnarray}
Hence, we assume that \eqref{e1.11} and \eqref{e1.12} both hold for
less than $m-1$. For $m$ case, we have,
\begin{eqnarray}
&&D_m(i_1\cdots, i_m)\\
&=&\sum_{l=1}^{m-1}(-1)^{m+l}a_{i_mi_l}B_{m-1}(i_1\cdots i_m;i_l)+a_{i_mi_m}D_{m-1}(i_1\cdots i_{m-1})\nonumber\\
&=&\sigma_{n-2}^{m-3}(\la|1)[\sigma_{n-2}(\la|1)\sigma_{n-3}(\la|1i_m) \sigma_{n-(m+1)}(\la|1i_1\cdots i_{m-1})\nonumber\\
&&-\sum_{l=1}^{m-1}\sigma^2_{n-3}(\la|1i_li_m) \sigma_{n-m}(\la|1i_1\cdots i_{l-1}i_{l+1}\cdots i_{m-1})\nonumber\\
&&-\sum_{l=1}^{m-1}\sigma_{n-3}(\la|1i_li_m) \sigma_{n-m}(\la|1i_2\cdots i_m)\sum_{k\neq l,m}\sigma_{n-3}(\la|1i_1i_k)]\nonumber.
\end{eqnarray}
We also have,
\begin{eqnarray}
&&\sigma_{n-3}(\la|1i_li_m) \sigma_{n-m}(\la|1i_1\cdots i_{l-1}i_{l+1}\cdots i_{m-1})\\&&
+ \sigma_{n-m}(\la|1i_2\cdots i_m)\sum_{k\neq l,m}\sigma_{n-3}(\la|1i_1i_k)\nonumber\\
&=&\sigma_{n-3}(\la|1i_li_m)(\la_{i_l} \sigma_{n-m-1}(\la|1i_1\cdots i_{m-1})+\sigma_{n-m}(\la|1i_1\cdots i_{m-1}))\nonumber\\&&
+ \sigma_{n-m}(\la|1i_1\cdots i_{m-1})\sum_{k\neq l}\sigma_{n-3}(\la|1i_ki_m)\nonumber\\
&=&\sigma_{n-2}(\la|1i_m) \sigma_{n-m-1}(\la|1i_1\cdots i_{m-1})
+ \sigma_{n-m}(\la|1i_1\cdots i_{m-1})\sum_{k}\sigma_{n-3}(\la|1i_ki_m)\nonumber\\
&=&\sigma_{n-3}(\la|1i_1i_m)(\la_{i_1} \sigma_{n-m-1}(\la|1i_1\cdots i_{m-1})+\sigma_{n-m}(\la|1i_1\cdots i_{m-1}))\nonumber\\&&
+ \sigma_{n-m}(\la|1i_1\cdots i_{m-1})\sum_{k\neq 1}\sigma_{n-3}(\la|1i_ki_m)\nonumber\\
&=&\sigma_{n-3}(\la|1i_1i_m) \sigma_{n-m}(\la|1i_2\cdots i_{m-1})\nonumber\\
&&+ \sigma_{n-m+1}(\la|1i_2\cdots i_{m-1})\sum_{k}\sigma_{n-4}(\la|1i_1i_ki_m)\nonumber\\
&=&\sigma_{n-4}(\la|1i_1i_2i_m)(\la_{i_2} \sigma_{n-m}(\la|1i_2\cdots i_{m-1})+\sigma_{n-m+1}(\la|1i_2\cdots i_{m-1}))\nonumber\\&&
+ \sigma_{n-m+1}(\la|1i_2\cdots i_{m-1})\sum_{k\neq 2}\sigma_{n-4}(\la|1i_1i_ki_m)\nonumber\\
&=&\sigma_{n-4}(\la|1i_1i_2i_m) \sigma_{n-m+1}(\la|1i_3\cdots i_{m-1})\nonumber\\
&&+ \sigma_{n-m+2}(\la|1i_3\cdots i_{m-1})\sum_{k}\sigma_{n-5}(\la|1i_1i_2i_ki_m)\nonumber\\
&=&\cdots =\sigma_{n-(m+1)}(\la|1i_1\cdots i_m)\sigma_{n-2}(\la|1).\nonumber
\end{eqnarray}
Hence, we have,
\begin{eqnarray}
&&D_m(i_1\cdots, i_m)\\
&=&\sigma_{n-2}^{m-2}(\la|1)[\sigma_{n-3}(\la|1i_m) \sigma_{n-(m+1)}(\la|1i_1\cdots i_{m-1})\nonumber\\
&&-\sigma_{n-m-1}(\la|1i_1\cdots i_m)\sum_{l=1}^{m-1}\sigma_{n-3}(\la|1i_li_m) \nonumber\\
&=&\sigma_{n-2}^{m-2}(\la|1)[\sigma_{n-3}(\la|1i_m)\la_{i_m} \sigma_{n-(m+2)}(\la|1i_1\cdots i_{m})\nonumber\\
&&+\sigma_{n-m-1}(\la|1i_1\cdots i_m)(\sigma_{n-3}(\la|1i_m)-\sum_{l=1}^{m-1}\sigma_{n-3}(\la|1i_li_m))]. \nonumber
\end{eqnarray}
It is clear that we have,
\begin{eqnarray}
&&\sigma_{n-3}(\la|1i_m)-\sum_{l=1}^{m-1}\sigma_{n-3}(\la|1i_li_m)\\
&=&\la_{i_1}\sigma_{n-4}(\la|1i_1i_m)-\sum_{l=2}^{m-1}\sigma_{n-3}(\la|1i_li_m)\nonumber\\
&=&\la_{i_1}[\sigma_{n-4}(\la|1i_1i_m)-\sum_{l=2}^{m-1}\sigma_{n-4}(\la|1i_1i_li_m)]\nonumber\\
&=&\la_{i_1}[\la_{i_2}\sigma_{n-5}(\la|1i_1i_2i_m)-\sum_{l=3}^{m-1}\sigma_{n-4}(\la|1i_1i_li_m)]\nonumber\\
&=&\la_{i_1}\la_{i_2}[\sigma_{n-5}(\la|1i_1i_2i_m)-\sum_{l=3}^{m-1}\sigma_{n-5}(\la|1i_1i_2i_li_m)]\nonumber\\
&=&\cdots =\la_{i_1}\la_{i_2}\cdots\la_{i_{m-1}}\sigma_{n-(m+2)}(\la|1i_1\cdots i_m)\nonumber.
\end{eqnarray}
Hence, we obtain,
\begin{eqnarray}
&&D_m(i_1\cdots, i_m)\nonumber\\
&=&\sigma_{n-2}^{m-2}(\la|1)[\sigma_{n-3}(\la|1i_m)\la_{i_m} \sigma_{n-(m+2)}(\la|1i_1\cdots i_{m})\nonumber\\
&&+\sigma_{n-m-1}(\la|1i_1\cdots i_m)\la_{i_1}\la_{i_2}\cdots\la_{i_{m-1}}\sigma_{n-(m+2)}(\la|1i_1\cdots i_m)] \nonumber\\
&=&\sigma_{n-2}^{m-2}(\la|1)[\sigma_{n-3}(\la|1i_m)\la_{i_m}+\sigma_{n-2}(\la|1 i_m)]\sigma_{n-(m+2)}(\la|1i_1\cdots i_m) \nonumber\\
&=&\sigma_{n-2}^{m-1}(\la|1)\sigma_{n-(m+2)}(\la|1i_1\cdots i_m) \nonumber.
\end{eqnarray}
For the formula \eqref{e1.12}, we can rewrite it to be,
\begin{eqnarray}\label{a.18}
&&B_{m-1}(i_1,\cdots,i_m; i_{k})\\
&=&(-1)^{m+k}[\sigma_{n-3}(\la|1i_ki_m)D_{m-2}(i_1\cdots i_{k-1}i_{k+1}\cdots i_{m-1})\nonumber\\&&+\sigma_{n-m}(\la|1i_1\cdots i_{m-1})\sigma_{n-2}^{m-3}(\la|1)\sum_{l\neq k}\sigma_{n-3}(\la|1i_l i_m)]\nonumber.
\end{eqnarray}
Now, let's expand its last row to prove it. In what following,  $\hat{i}$ means that index $i$ does not appear. We have,
\begin{eqnarray}\label{a.19}
&&B_{m-1}(i_1\cdots i_m;i_k)\\
&=&\sum_{l>k}(-1)^{m+l}\sigma_{n-3}(\la|1i_li_m)(-1)^{m-2-l}B_{m-2}(i_1\cdots i_{l-1}i_{l+1}\cdots i_{m-1} i_l;i_k)\nonumber\\
&&+\sum_{l<k}(-1)^{m+l}\sigma_{n-3}(\la|1i_li_m)(-1)^{m-1-l}B_{m-2}(i_1\cdots i_{l-1}i_{l+1}\cdots i_{m-1} i_l;i_k)\nonumber\\
&&+(-1)^{m+k}\sigma_{n-3}(\la|1i_ki_m)D_{m-2}(i_1\cdots i_{k-1}i_{k+1}\cdots i_{m-1})\nonumber\\
&=&(-1)^{m-1+k-1}\sum_{l>k}\sigma_{n-3}(\la|1i_li_m)[\sigma_{n-3}(\la|1i_ki_l)D_{m-3}(i_1\cdots \hat{i}_l\cdots \hat{i}_k\cdots i_{m-1})\nonumber\\
&&+\sigma_{n-(m-1)}(\la|1i_1\cdots\hat{i}_l\cdots i_{m-1})\sigma_{n-2}^{m-4}(\la|1)\sum_{a\neq k}\sigma_{n-3}(\la|1i_li_a)]\nonumber\\
&&+(-1)(-1)^{m-1+k}\sum_{l<k}\sigma_{n-3}(\la|1i_li_m)[\sigma_{n-3}(\la|1i_ki_l)D_{m-3}(i_1\cdots \hat{i}_k\cdots \hat{i}_l\cdots i_{m-1})\nonumber\\
&&+\sigma_{n-(m-1)}(\la|1i_1\cdots\hat{i}_l\cdots i_{m-1})\sigma_{n-2}^{m-4}(\la|1)\sum_{a\neq k}\sigma_{n-3}(\la|1i_li_a)]\nonumber\\
&&+(-1)^{m+k}\sigma_{n-3}(\la|1i_ki_m)D_{m-2}(i_1\cdots i_{k-1}i_{k+1}\cdots i_{m-1})\nonumber\\
&=&(-1)^{m+k}\{\sum_{l\neq k}\sigma_{n-3}(\la|1i_li_m)[\sigma_{n-3}(\la|1i_ki_l)D_{m-3}(i_1\cdots \hat{i}_l\cdots \hat{i}_k\cdots i_{m-1})\nonumber\\
&&+\sigma_{n-(m-1)}(\la|1i_1\cdots\hat{i}_l\cdots i_{m-1})\sigma_{n-2}^{m-4}(\la|1)\sum_{a\neq k}\sigma_{n-3}(\la|1i_li_a)]\nonumber\\
&&+\sigma_{n-3}(\la|1i_ki_m)D_{m-2}(i_1\cdots i_{k-1}i_{k+1}\cdots i_{m-1})\}\nonumber\\
&=&(-1)^{m+k}\{\sigma^{m-4}_{n-2}(\la|1)\sum_{l\neq k}\sigma_{n-3}(\la|1i_li_m)\nonumber\\
&&\times [\sigma_{n-3}(\la|1i_ki_l)\sigma_{n-(m-1)}(\la|1i_1\cdots \hat{i}_l\cdots \hat{i}_k\cdots i_{m-1})\nonumber\\
&&+\sigma_{n-(m-1)}(\la|1i_1\cdots\hat{i}_l\cdots i_{m-1})\sum_{a\neq k}\sigma_{n-3}(\la|1i_li_a)]\nonumber\\
&&+\sigma_{n-3}(\la|1i_ki_m)D_{m-2}(i_1\cdots i_{k-1}i_{k+1}\cdots i_{m-1})\}\nonumber.
\end{eqnarray}
We see that,
\begin{eqnarray}\label{a.20}
&&\sigma_{n-3}(\la|1i_ki_l)\sigma_{n-(m-1)}(\la|1i_1\cdots \hat{i}_l\cdots \hat{i}_k\cdots i_{m-1})\\
&&+\sigma_{n-(m-1)}(\la|1i_1\cdots\hat{i}_l\cdots i_{m-1})\sum_{a\neq k}\sigma_{n-3}(\la|1i_li_a)\nonumber\\
&=&\sigma_{n-2}(\la|1i_l)\sigma_{n-m}(\la|1i_1\cdots \hat{i}_l\cdots i_{m-1})\nonumber\\
&&+\sigma_{n-(m-1)}(\la|1i_1\cdots\hat{i}_l\cdots i_{m-1})\sum_{a}\sigma_{n-3}(\la|1i_li_a)\nonumber
\end{eqnarray}
\begin{eqnarray}
&=&\sigma_{n-3}(\la|1i_1i_l)[\la_{i_1}\sigma_{n-m}(\la|1i_1\cdots \hat{i}_l\cdots i_{m-1})+\sigma_{n-(m-1)}(\la|1i_1\cdots\hat{i}_l\cdots i_{m-1})]\nonumber\\
&&+\sigma_{n-(m-1)}(\la|1i_1\cdots\hat{i}_l\cdots i_{m-1})\sum_{a\neq 1}\sigma_{n-3}(\la|1i_li_a)\nonumber\\
&=&\sigma_{n-3}(\la|1i_1i_l)\sigma_{n-(m-1)}(\la|1i_2\cdots\hat{i}_l\cdots i_{m-1})\nonumber\\
&&+\sigma_{n-m+2}(\la|1i_2\cdots\hat{i}_l\cdots i_{m-1})\sum_{a}\sigma_{n-4}(\la|1i_1i_li_a)\nonumber\\
&=&\sigma_{n-4}(\la|1i_1i_2i_l)[\la_{i_2}\sigma_{n-m+1}(\la|1i_2\cdots \hat{i}_l\cdots i_{m-1})+\sigma_{n-m+2}(\la|1i_2\cdots\hat{i}_l\cdots i_{m-1})]\nonumber\\
&&+\sigma_{n-m+2}(\la|1i_2\cdots\hat{i}_l\cdots i_{m-1})\sum_{a\neq 2}\sigma_{n-4}(\la|1i_1i_li_a)\nonumber\\
&=&\sigma_{n-4}(\la|1i_1i_2i_l)\sigma_{n-m+2}(\la|1i_3\cdots\hat{i}_l\cdots i_{m-1})\nonumber\\
&&+\sigma_{n-m+3)}(\la|1i_3\cdots\hat{i}_l\cdots i_{m-1})\sum_{a}\sigma_{n-5}(\la|1i_1i_2i_li_a)\nonumber\\
&=&\cdots =\sigma_{n-m}(\la|1i_1\cdots i_{m-1})\sigma_{n-2}(\la|1)\nonumber.
\end{eqnarray}
Hence, combing \eqref{a.19} and \eqref{a.20},  we obtain \eqref{a.18}.
\end{proof}
At last, we give a counter example. This example says that our inequality holds only for $\sigma_{n-1}$. We consider the $\sigma_2$ in dimension $4$. Suppose $$\la_1=2t+\frac{1}{t}, \la_2=2t, \la_3=0,\text{ and } \la_4=-t.$$ Then, a directly calculate gives,
$$\sigma_2^{11}=t,\sigma_2^{22}=t+\frac{1}{t}, \sigma_2^{33}=3t+\frac{1}{t},\sigma_2^{44}=4t+\frac{1}{t}, \sigma_2=1.$$ Hence, $(\kappa_1,\kappa_2,\kappa_3,\kappa_4)$ is in $\Gamma_2$ cone for $t>0$. Let's calculate the determinate of the matrix defined by the bilinear form \eqref{e1.1} for $i=1$ case. It is equal to,
$$275K-311+\frac{12K-12}{t^4}+\frac{96K-100}{t^2}+(313K-427)t^2+(66K-216)t^4-72Kt^6.$$
Obviously, it is not nonnegative for sufficient large $t$.

\section{Global  curvature estimate }
\par
In this section, we consider the global $C^2$-estimates for the
curvature equation of $k=n-1$. At first, we need the following Lemma.

\begin{lemm} \label{leR}    %%%%%%%%%%%%%%%%%%%%%%%%   lemma 3
 For any constant
$0<\varepsilon_T<\dfrac{1}{2}$, there exist another constant
$0<\delta<\min\{\varepsilon_T/2,1/200\}$, which depends on $\varepsilon_T$, such that,
if $|\kappa_i|<\delta\kappa_1$,  we have,
\begin{align}\label{e2.2}
(1+\varepsilon_T)e^{\kappa_l}\sigma_{k-2}(\kappa|il)+(1+\varepsilon_T)\dfrac{e^{\kappa_l}-e^{\kappa_i}}{\kappa_l-\kappa_i}
\sigma_{k-1}(\kappa|l)\geq
\dfrac{e^{\kappa_l}}{\kappa_1}\sigma_{k-1}(\kappa|i),
\end{align}
for sufficient large  $\kappa_1$.\end{lemm}
\begin{proof}  It is obvious that we have the following identity,
$$
\sigma_{k-1}(\kappa|l)=\sigma_{k-1}(\kappa|i)+(\kappa_i-\kappa_l)\sigma_{k-2}(\kappa|il).
$$
Multiplying $\dfrac{e^{\kappa_l}-e^{\kappa_i}}{\kappa_l-\kappa_i}$ in  both side of the above identity, we have,  
\begin{align}\label{e2.3}
e^{\kappa_l}\sigma_{k-2}(\kappa|il)+\dfrac{e^{\kappa_l}-e^{\kappa_i}}{\kappa_l-\kappa_i}
\sigma_{k-1}(\kappa|l)
=e^{\kappa_i}\sigma_{k-2}(\kappa|il)+\dfrac{e^{\kappa_l}-e^{\kappa_i}}{\kappa_l-\kappa_i}
\sigma_{k-1}(\kappa|i).
\end{align}
We divide into four cases to discuss.
\par
\noindent Case (i): $\kappa_l\leq \kappa_i$.\\
In this case, we have,
\begin{align*}
\dfrac{e^{\kappa_l}-e^{\kappa_i}}{\kappa_l-\kappa_i}
\sigma_{k-1}(\kappa|i)
=e^{\kappa_l}\dfrac{e^{\kappa_i-\kappa_l}-1}{\kappa_i-\kappa_l}
\sigma_{k-1}(\kappa|i)\geq e^{\kappa_l}\sigma_{k-1}(\kappa|i).
\end{align*}
Hence, by (\ref{e2.3}),  we get (\ref{e2.2})  for sufficient large $\kappa_1$. 
\par
\noindent Case (ii): $0<\kappa_l-\kappa_i\leq 1$.\\
In this case, obviously, we have $\kappa_i\geq\kappa_l-1$.  By the mean value theorem, there exists some constant
$\kappa_i<\xi<\kappa_l$. Then we have, 
\begin{align*}
\dfrac{e^{\kappa_l}-e^{\kappa_i}}{\kappa_l-\kappa_i}
\sigma_{k-1}(\kappa|i) =e^{\xi}\sigma_{k-1}(\kappa|i)\geq
e^{\kappa_l-1}\sigma_{k-1}(\kappa|i)
\geq\dfrac{e^{\kappa_l}}{\kappa_1}\sigma_{k-1}(\kappa|i),
\end{align*}
if $\kappa_1$ is sufficient large. By (\ref{e2.3}), we get
(\ref{e2.2}).
\par
\noindent Case(iii): $\kappa_l-\kappa_i>1$ and
$\dfrac{\kappa_l}{\kappa_1}\leq\dfrac{1}{100}$.\\
Using the condition $|\kappa_i|<\delta\kappa_1$, we have,
$$\kappa_l-\kappa_i\leq (\delta +\dfrac{1}{100})\kappa_1.$$ Then, we have,
\begin{align*}
\dfrac{e^{\kappa_l}-e^{\kappa_i}}{\kappa_l-\kappa_i}
\sigma_{k-1}(\kappa|i) \geq
e^{\kappa_l}\dfrac{1-e^{-1}}{\kappa_l-\kappa_i}
\sigma_{k-1}(\kappa|i) \geq
\dfrac{1-e^{-1}}{\frac{1}{100}+\delta}\dfrac{e^{\kappa_l}}{\kappa_1}
\sigma_{k-1}(\kappa|i).
\end{align*}
Now, choosing  $\delta$ sufficient small, we get,
$$\dfrac{1-e^{-1}}{\frac{1}{100}+\delta}\geq 1.$$ Then insert the above two inequalities into (\ref{e2.3}),  we
get (\ref{e2.2}).

\noindent Case (iv): $\kappa_l-\kappa_i>1$ and $\dfrac{\kappa_l}{\kappa_1}>\dfrac{1}{100}$.\\
In this case,  (\ref{e2.2}) can be rewritten, 
\begin{align}\label{e2.4}
&(1+\varepsilon_T)e^{\kappa_l}\sigma_{k-2}(\kappa|il)
+(1+\varepsilon_T)\dfrac{e^{\kappa_l}-e^{\kappa_i}}{\kappa_l-\kappa_i}
(\kappa_i\sigma_{k-2}(\kappa|il)+\sigma_{k-1}(\kappa|il))\\
\geq
&\dfrac{e^{\kappa_l}}{\kappa_1}(\kappa_l\sigma_{k-2}(\kappa|il)+\sigma_{k-1}(\kappa|il)).
\nonumber
\end{align}
If $\sigma_{k-1}(\kappa|il)\leq 0$,  (\ref{e2.4})  is clearly true. 
Thus, we can assume $\sigma_{k-1}(\kappa|il)> 0$. Obviously, we have, 
$$
e^{\kappa_l}\sigma_{k-2}(\kappa|il)\geq\dfrac{e^{\kappa_l}}{\kappa_1}\kappa_l\sigma_{k-2}(\kappa|il).
$$
To prove  (\ref{e2.4}), we only need to show the following two
inequalities, 
\begin{align}\label{e2.5}
(1+\varepsilon_T)\dfrac{e^{\kappa_l}-e^{\kappa_i}}{\kappa_l-\kappa_i}
\sigma_{k-1}(\kappa|il)\geq
\dfrac{e^{\kappa_l}}{\kappa_1}\sigma_{k-1}(\kappa|il),
\end{align}
and
\begin{align}\label{e2.6}
\varepsilon_T e^{\kappa_l}\sigma_{k-2}(\kappa|il)
+(1+\varepsilon_T)\dfrac{e^{\kappa_l}-e^{\kappa_i}}{\kappa_l-\kappa_i}
\kappa_i\sigma_{k-2}(\kappa|il)\geq 0.
\end{align}
To obtain (\ref{e2.5}), since $\sigma_{k-1}(\kappa|il)>0$, we can take off it in both sides. Hence, we only need,  
$$
\varepsilon_T\kappa_1e^{\kappa_l}+\kappa_ie^{\kappa_l}-(1+\varepsilon_T)\kappa_1e^{\kappa_i}\geq
0.
$$
Using $|\kappa_i| < \delta\kappa_1$, we need, 
$$
(\varepsilon_T-\delta)\kappa_1e^{\kappa_l}-(1+\varepsilon_T)\kappa_1e^{\kappa_i}\geq
0,$$ which implies the following requirement, 
$$\kappa_l-\kappa_i\geq\log(\dfrac{1+\varepsilon_T}{\varepsilon_T-\delta}).
$$
Since $|\kappa_i|<\delta\kappa_1,
\kappa_l>\dfrac{1}{100}\kappa_1$, the above requirement can be satisfied, if 
$$(\dfrac{1}{100}-\delta)\kappa_1\geq\log(\dfrac{1+\varepsilon_T}{\varepsilon_T-\delta}).$$
Hence, taking sufficient large $\kappa_1$, we obtain the above inequality.

In order to get (\ref{e2.6}), we need,
$$
\varepsilon_T+(1+\varepsilon_T)\dfrac{1-e^{\kappa_i-\kappa_l}}{\kappa_l-\kappa_i}
\kappa_i\geq 0.
$$
If $\kappa_i\geq 0$, it is clearly right. Hence, we only consider the case $\kappa_i<0$.
Then, we need to require,
$$
 (\kappa_l-\kappa_i)\varepsilon_T\geq -(1+\varepsilon_T)\kappa_i,$$ which implies, 
 $$\kappa_l\varepsilon_T\geq -\kappa_i.
$$
By our assumption, $\kappa_l\geq\dfrac{1}{100}\kappa_1,
|\kappa_i|\leq \delta \kappa_1$, we only need the constants $\delta$ and $\varepsilon_T$ to satisfy, 
$$\dfrac{\varepsilon_T}{100}\geq\delta. $$ We complete our proof.

\end{proof}

\par
Now we consider the global $C^2$-estimates for the curvature
equation \eqref{1.1}.
\par
Set $u(X)=<X, \nu(X)>$. By the assumption that $M$ is starshaped
with a $C^1$ bound, $u$ is bounded from below and above by two
positive constants. At every point in the hypersurface $M$, choose a
local coordinate frame $\{ \p/(\p x_1),\cdots,\p/(\p x_{n+1})\}$ in
$\mathbb{R}^n$ such that the first $n$ vectors are the local
coordinates of the hypersurface and the last one is the unit outer
normal vector.  Denote $\nu$ to be the outer normal vector. We let
$h_{ij}$ and $u$ be the second fundamental form and the support
function of the hypersurface $M$ respectively.  The following
geometric formulas are well known (e.g., \cite{GLL}).

\begin{equation}
h_{ij}=\langle\partial_iX,\partial_j\nu\rangle,
\end{equation}
and
\begin{equation}
\begin{array}{rll}
X_{ij}=& -h_{ij}\nu\quad {\rm (Gauss\ formula)}\\
(\nu)_i=&h_{ij}\partial_j\quad {\rm (Weigarten\ equation)}\\
h_{ijk}=& h_{ikj}\quad {\rm (Codazzi\ formula)}\\
R_{ijkl}=&h_{ik}h_{jl}-h_{il}h_{jk}\quad {\rm (Gauss\ equation)},\\
\end{array}
\end{equation}
where $R_{ijkl}$ is the $(4,0)$-Riemannian curvature tensor. We also
have
\begin{equation}
\begin{array}{rll}
h_{ijkl}=& h_{ijlk}+h_{mj}R_{imlk}+h_{im}R_{jmlk}\\
=& h_{klij}+(h_{mj}h_{il}-h_{ml}h_{ij})h_{mk}+(h_{mj}h_{kl}-h_{ml}h_{kj})h_{mi}.\\
\end{array}
\end{equation}

\par
For function $u$, we consider the following test function which appear firstly in \cite{GRW},
$$
\phi=\log\log P-N\ln u.
$$
Here the function $P$ is defined by $$P=\dsum_le^{\kappa_l}. $$

We may assume that the maximum of $\phi$ is achieved  at some point
$X_0\in M$. After rotating the coordinates, we may assume the matrix
$(h_{ij})$ is diagonal at the point, and we can further  assume that
$h_{11}\geq h_{22}\cdots\geq h_{nn}$. Denote $\kappa_i=h_{ii}$.
\par
Differentiate the function twice  at  $X_0$, we have,
\begin{equation}\label{e2.11}
\phi_i=\dfrac{P_i}{P\log P}- N\frac{h_{ii}\langle
X,\p_i\rangle}{u}=0,
\end{equation}

and,
\begin{eqnarray*}
&&\phi_{ii}\\
&=& \frac{P_{ii}}{P\log P}-\frac{P_i^2}{P^2\log P}-\frac{P_i^2}{(P\log P)^2}- \frac{N}{u}\sum_lh_{il,i}\langle \p_l,X \rangle-\frac{N h_{ii}}{u}\nonumber\\&&+Nh_{ii}^2+N\frac{h_{ii}^2\langle X,\p_i\rangle^2}{u^2}\nonumber\\
&=&\frac{1}{P\log P}[\sum_le^{\kappa_l}h_{llii}+\sum_le^{\kappa_l}h_{lli}^2+\sum_{\alpha\neq \beta}\frac{e^{\kappa_{\alpha}}-e^{\kappa_{\beta}}}{\kappa_{\alpha}-\kappa_{\beta}}h_{\alpha\beta i}^2-(\frac{1}{P}+\frac{1}{P\log P})P_i^2]\nonumber\\
&&- \frac{N\sum_lh_{iil}\langle \p_l,X \rangle}{u}-\frac{ Nh_{ii}}{u}+Nh_{ii}^2+N\frac{h_{ii}^2\langle X,\p_i\rangle^2}{u^2}\nonumber\\
&=&\frac{1}{P\log P}[\sum_le^{\kappa_l}h_{ii,ll}+\sum_le^{\kappa_l}(h_{il}^2-h_{ii}h_{ll})h_{ii}+\sum_le^{\kappa_l}(h_{ii}h_{ll}-h_{il}^2)h_{ll}\nonumber\\
&&+\sum_le^{\kappa_l}h_{lli}^2+\sum_{\alpha\neq \beta}\frac{e^{\kappa_{\alpha}}-e^{\kappa_{\beta}}}{\kappa_{\alpha}-\kappa_{\beta}}h_{\alpha\beta i}^2-(\frac{1}{P}+\frac{1}{P\log P})P_i^2]\nonumber\\
&&- \frac{N\sum_lh_{iil}\langle \p_l,X \rangle}{u}-\frac{
Nh_{ii}}{u}+Nh_{ii}^2+N\frac{h_{ii}^2\langle
X,\p_i\rangle^2}{u^2}\nonumber
\end{eqnarray*}
Contract with $\sigma_{n-1}^{ii}$,
\begin{eqnarray}\label{e2.13}
&&\sigma_{n-1}^{ii}\phi_{ii}\\
&=&\frac{1}{P\log
P}[\sum_le^{\kappa_l}\sigma_{n-1}^{ii}h_{ii,ll}+(n-1)f\sum_le^{\kappa_l}h_{ll}^2-\sigma_{n-1}^{ii}h_{ii}^2\sum_le^{\kappa_l}h_{ll}
\nonumber\\
&&+\sum_l\sigma_{n-1}^{ii}e^{\kappa_l}h_{lli}^2+\sum_{\alpha\neq \beta}\sigma_{n-1}^{ii}\frac{e^{\kappa_{\alpha}}-e^{\kappa_{\beta}}}{\kappa_{\alpha}-\kappa_{\beta}}h_{\alpha\beta i}^2-(\frac{1}{P}+\frac{1}{P\log P})\sigma_{n-1}^{ii}P_i^2]\nonumber\\
&&- \frac{N\sum_l\sigma_{n-1}^{ii}h_{iil}\langle \p_l,X
\rangle}{u}-\frac{
N(n-1)f}{u}+N\sigma_{n-1}^{ii}h_{ii}^2+N\frac{\sigma_{n-1}^{ii}h_{ii}^2\langle
X,\p_i\rangle^2}{u^2}.\nonumber
\end{eqnarray}\par
At $x_0$, differentiate equation (\ref{1.1}) twice, we have,
\begin{eqnarray}\label{e2.14}
\sigma_{n-1}^{ii}h_{iik}&=&d_Xf(\p_k)+h_{kk}d_{\nu}f(\p_k),
\end{eqnarray}
and
\begin{eqnarray}\label{e2.15}
\sigma_{n-1}^{ii}h_{iikk}+\sigma_{n-1}^{pq,rs}h_{pqk}h_{rsk}&\geq&
-C-Ch_{11}^2+\sum_lh_{lkk}d_{\nu}f(\p_l),
\end{eqnarray}
where $C$ is some constant under control.

Insert  (\ref{e2.15}) into (\ref{e2.13}),
\begin{eqnarray}
&&\sigma_{n-1}^{ii}\phi_{ii}\label{e2.16}\\
&\geq &\frac{1}{P\log P}[\sum_le^{\kappa_l}(-C-Ch_{11}^2-\sigma_{n-1}^{pq,rs}h_{pql}h_{rsl})+\sum_le^{\kappa_k}h_{lkk}d_{\nu}f(\p_l)\nonumber\\
&&+(n-1)f\sum_le^{\kappa_l}h_{ll}^2-\sigma_{n-1}^{ii}h_{ii}^2\sum_le^{\kappa_l}h_{ll}+\sum_l\sigma_{n-1}^{ii}e^{\kappa_l}h_{lli}^2\nonumber\\
&&+\sum_{\alpha\neq \beta}\sigma_{n-1}^{ii}\frac{e^{\kappa_{\alpha}}-e^{\kappa_{\beta}}}{\kappa_{\alpha}-\kappa_{\beta}}h_{\alpha\beta i}^2-(\frac{1}{P}+\frac{1}{P\log P})\sigma_{n-1}^{ii}P_i^2]\nonumber\\
&&- \frac{N\sum_l\sigma_{n-1}^{ii}h_{iil}\langle \p_l,X
\rangle}{u}-\dfrac{N(n-1)f}{u}+N\sigma_{n-1}^{ii}h_{ii}^2+N\frac{\sigma_{n-1}^{ii}h_{ii}^2\langle
X,\p_i\rangle^2}{u^2}.\nonumber
\end{eqnarray}
By (\ref{e2.11}) and (\ref{e2.14}), we have, 
\begin{eqnarray}\label{e2.17}
&&\sum_k d_{\nu}f(\p_k)\frac{\sum_le^{\kappa_l}h_{llk}}{P\log P}-\frac{N}{u}\sum_k\sigma_{n-1}^{ii}h_{iik}\langle  \p_k, X\rangle\\
&=&-\frac{N}{u}\sum_kd_Xf(\p_k)\langle X,\p_k\rangle.\nonumber
\end{eqnarray}

Denote
\begin{eqnarray}
&&A_i=e^{\kappa_i}(K(\sigma_{n-1})_i^2-\sum_{p\neq q}\sigma_{n-1}^{pp,qq}h_{ppi}h_{qqi}), \ \  B_i=2\sum_{l\neq i}\sigma_{n-1}^{ii,ll}e^{\kappa_l}h_{lli}^2, \nonumber \\  &&C_i=\sigma_{n-1}^{ii}\sum_le^{\kappa_l}h_{lli}^2; \  \ D_i=2\sum_{l\neq
i}\sigma_{n-1}^{ll}\frac{e^{\kappa_l}-e^{\kappa_i}}{\kappa_l-\kappa_i}h_{lli}^2,
\ \ E_i=\frac{1+\log P}{P\log P}\sigma_{n-1}^{ii}P_i^2\nonumber.
\end{eqnarray}
Using
$$
-\dsum_l\sigma_{n-1}^{pq,rs}h_{pql}h_{rsl}=\dsum_{p\neq
q}\sigma_{n-1}^{pp,qq}h_{pql}^2-\dsum_{p\neq
q}\sigma_{n-1}^{pp,qq}h_{ppl}h_{qql},
$$
and (\ref{e2.16}), for any $K>1$, we have,
\begin{eqnarray}
&&\sigma_{n-1}^{ii}\phi_{ii}\label{e2.18}\\
&\geq &\frac{1}{P\log P}[\sum_le^{\kappa_l}(K(\sigma_{n-1})_l^2-\dsum_{p\neq q}\sigma_{n-1}^{pp,qq}h_{ppl}h_{qql}+\dsum_{p\neq q}\sigma_{n-1}^{pp,qq}h_{pql}^2)\nonumber\\
&&+\sum_l\sigma_{n-1}^{ii}e^{\kappa_l}h_{lli}^2+\sum_{\alpha\neq
\beta}\sigma_{n-1}^{ii}\frac{e^{\kappa_{\alpha}}-e^{\kappa_{\beta}}}{\kappa_{\alpha}-\kappa_{\beta}}h_{\alpha\beta
i}^2-\frac{1+\log P}{P\log
P}\sigma_{n-1}^{ii}P_i^2\nonumber\\
&&-CP-CKPh_{11}^2]+(N-1)\sigma_{n-1}^{ii}h_{ii}^2+N\frac{\sigma_{n-1}^{ii}h_{ii}^2\langle
X,\p_i\rangle^2}{u^2}\nonumber\\
&&\geq \frac{1}{P\log P}\dsum_i(A_i+B_i+C_i+D_i-E_i)\nonumber\\
&&+(N-1)\sigma_{n-1}^{ii}h_{ii}^2+N\frac{\sigma_{n-1}^{ii}h_{ii}^2\langle
X,\p_i\rangle^2}{u^2}-\dfrac{C+CKh_{11}^2}{\log P}\nonumber.
\end{eqnarray}

\begin{lemm}\label{le2}
There exists a constant $\delta<\dfrac{1}{2}$ such that, if
$|\kappa_i|\leq \delta \kappa_1$, we have,
\begin{align*}
A_i+B_i+C_i+D_i-E_i\geq0,
\end{align*}
for sufficient large $K$ and $\kappa_1$.
\end{lemm}
\begin{proof} Firstly, using Lemma \ref{Guan}, we have $A_i> 0$,
for sufficient large constant $K$.
By Cauchy-Schwarz inequality, we have,
\begin{eqnarray}\label{e2.19}
P_i^2&= & e^{2\kappa_i}h_{iii}^2 +2 \dsum_{l\neq
i}e^{\kappa_i+\kappa_l}h_{iii}h_{lli}
+ (\dsum_{l\neq i}e^{\kappa_l}h_{lli})^2\\
&\leq & e^{2\kappa_i}h_{iii}^2 +2 \dsum_{l\neq
i}e^{\kappa_i+\kappa_l}h_{iii}h_{lli}+ (P-e^{\kappa_i})\dsum_{l\neq
i}e^{\kappa_l}h_{lli}^2. \nonumber
\end{eqnarray}
Using (\ref{e2.19}),  we have,
\begin{eqnarray}\label{e2.20}
&&B_i+C_i+D_i-E_i\\
&\geq&2\dsum_{l\neq i}e^{\kappa_l}\sigma_{n-1}^{ll,ii}h_{lli}^2 +
2\dsum_{l\neq
i}\dfrac{e^{\kappa_l}-e^{\kappa_i}}{\kappa_l-\kappa_i}\sigma_{n-1}^{ll}h_{lli}^2
-\dfrac{1}{\log P}\dsum_{l\neq
i}e^{\kappa_l}\sigma_{n-1}^{ii}h_{lli}^2\nonumber\\
&&+\frac{1+\log P}{P\log P}\dsum_{l\neq
i}e^{\kappa_l+\kappa_i}\sigma_{n-1}^{ii}h_{lli}^2+e^{\kappa_i}\sigma_{n-1}^{ii}h_{iii}^2 \nonumber\\
&&-\frac{1+\log P}{P\log P}e^{2\kappa_i}\sigma_{n-1}^{ii}h_{iii}^2
-2\frac{1+\log P}{P\log P}\dsum_{l\neq
i}e^{\kappa_i+\kappa_l}\sigma_{n-1}^{ii}h_{iii}h_{lli}  \nonumber.
\end{eqnarray}
Using Lemma \ref{leR}, there exists a constant
$\delta<\dfrac{1}{2}$, such that,
\begin{align}\label{e2.21}
\dfrac{3}{2}\dsum_{l\neq i}e^{\kappa_l}\sigma_{n-1}^{ll,ii}h_{lli}^2
+ \dfrac{3}{2}\dsum_{l\neq
i}\dfrac{e^{\kappa_l}-e^{\kappa_i}}{\kappa_l-\kappa_i}\sigma_{n-1}^{ll}h_{lli}^2
-\dfrac{1}{\log P}\dsum_{l\neq
i}e^{\kappa_l}\sigma_{n-1}^{ii}h_{lli}^2\geq 0.
\end{align}
On the other hand, we see that,
\begin{align}\label{CauS}
&\dsum_{l\neq i,1}e^{\kappa_l+\kappa_i}\sigma_{n-1}^{ii}h_{lli}^2
-2\dsum_{l\neq
i,1}e^{\kappa_i+\kappa_l}\sigma_{n-1}^{ii}h_{iii}h_{lli}\geq-\dsum_{l\neq
i,1}e^{\kappa_l+\kappa_i}\sigma_{n-1}^{ii}h_{iii}^2.
\end{align}
Then, using the above two inequalities, \eqref{e2.20} becomes, 
\begin{eqnarray}\label{e2.23}
&&B_i+C_i+D_i-E_i\\
&\geq&\frac{1+\log P}{P\log
P}e^{\kappa_1+\kappa_i}\sigma_{n-1}^{ii}h_{11i}^2+e^{\kappa_i}\sigma_{n-1}^{ii}h_{iii}^2 \nonumber\\
&&-\frac{1+\log P}{P\log P}\dsum_{l\neq
1}e^{\kappa_l+\kappa_i}\sigma_{n-1}^{ii}h_{iii}^2
-2\frac{1+\log P}{P\log P}e^{\kappa_i+\kappa_1}\sigma_{n-1}^{ii}h_{iii}h_{11i} \nonumber\\
&&+\dfrac{1}{2}e^{\kappa_1}\sigma_{n-1}^{11,ii}h_{11i}^2 +
\dfrac{1}{2}\dfrac{e^{\kappa_1}-e^{\kappa_i}}{\kappa_1-\kappa_i}\sigma_{n-1}^{11}h_{11i}^2.
\nonumber
\end{eqnarray}

Directly calculation shows that, 
\begin{align*}
e^{\kappa_i}\sigma_{n-1}^{ii}h_{iii}^2 -\frac{1+\log P}{P\log
P}\dsum_{l\neq
1}e^{\kappa_l+\kappa_i}\sigma_{n-1}^{ii}h_{iii}^2&\geq
(\frac{e^{\kappa_1}}{P}-\frac{1}{\log
P})e^{\kappa_i}\sigma_{n-1}^{ii}h_{iii}^2 \\
&\geq \frac{1}{n+1}e^{\kappa_i}\sigma_{n-1}^{ii}h_{iii}^2
\end{align*}
and
\begin{align*}
-2\frac{1+\log P}{P\log
P}e^{\kappa_i+\kappa_1}\sigma_{n-1}^{ii}|h_{iii}h_{11i}|\geq
-\frac{3}{P}e^{\kappa_i+\kappa_1}\sigma_{n-1}^{ii}|h_{iii}h_{11i}|\geq
-3e^{\kappa_i}\sigma_{n-1}^{ii}|h_{iii}h_{11i}|,
\end{align*}
hold for sufficient large $\kappa_1$. We let $l=1,k=n-1$ in \eqref{e2.11}, we have,
\begin{align}\label{e2.24}
e^{\kappa_1}\sigma_{n-1}^{11,ii}h_{11i}^2 +
\dfrac{e^{\kappa_1}-e^{\kappa_i}}{\kappa_1-\kappa_i}\sigma_{n-1}^{11}h_{11i}^2
=e^{\kappa_i}\sigma_{n-1}^{11,ii}h_{11i}^2 +
\dfrac{e^{\kappa_1}-e^{\kappa_i}}{\kappa_1-\kappa_i}\sigma_{n-1}^{ii}h_{11i}^2.
\end{align}
By Taylor expansion, we also have, 
\begin{align}\label{e2.25}
\dfrac{e^{\kappa_1}-e^{\kappa_i}}{\kappa_1-\kappa_i}\sigma_{n-1}^{ii}h_{11i}^2
=e^{\kappa_i}\dsum_{m\geq
1}\dfrac{(\kappa_1-\kappa_i)^{m-1}}{m!}\sigma_{n-1}^{ii}h_{11i}^2.
\end{align}
Combining the previous four formulas and using (\ref{e2.23}), we obtain, 
\begin{align}
B_i+C_i+D_i-E_i\geq&e^{\kappa_i}\sigma_{n-1}^{ii}[\frac{1}{n+1}h_{iii}^2-3|h_{iii}h_{11i}|
+\frac{1}{2}\dsum_{m\geq
1}\dfrac{(\kappa_1-\kappa_i)^{m-1}}{m!}h_{11i}^2] \geq 0,\nonumber
\end{align}
for sufficient large $\kappa_1$. 
\end{proof}
In $\Gamma_{n-1}$ cone, it is well known that the only possible negative eigenvalue is the smallest one. Since we have assumed that 
$\kappa_1\geq \kappa_2\geq \cdots\geq \kappa_n$, the possible non positive eigenvalue is $\kappa_n$. Hence, we can state the following little Lemma. 
\begin{lemm}\label{le6}
In $\Gamma_{n-1}$ cone, if $\kappa_n\leq 0$,
we have,  $$-\kappa_n\leq\frac{\kappa_1}{n-1}.$$ 
\end{lemm}
\begin{proof}
It is easy to see that, 
$$\sigma_{n-1}(\kappa|n)=\kappa_1\cdots\kappa_{n-1},\text{ and }
\sigma_{n-2}(\kappa|1n)=\kappa_2\cdots\kappa_{n-1}.$$ We assume that
$\lambda=-\kappa_n/\kappa_1$. Then we have,
\begin{align*}
\kappa_1\cdots\kappa_{n-1}=&\sigma_{n-1}-\kappa_n\sigma_{n-2}(\kappa|n)\\
> &-\kappa_n\sigma_{n-2}(\kappa|n)=
\lambda\kappa_1\sigma_{n-2}(\kappa|n)\\
=&\lambda\kappa_1^2\sigma_{n-3}(\kappa|n1)+\lambda\kappa_1\sigma_{n-2}(\kappa|1n).
\end{align*}
Hence, we get,
\begin{align*}
(1-\lambda)\kappa_2\cdots\kappa_{n-1}
> &\lambda\kappa_1\sigma_{n-3}(\kappa|n1)\geq (n-2)\lambda\kappa_2\cdots\kappa_{n-1},
\end{align*}
which implies $\lambda<\dfrac{1}{n-1}$.
\end{proof}

\begin{lemm}\label{le1}
For the chosen constant $\delta$ in Lemma \ref{le2}, if $\kappa_i\geq \delta
\kappa_1$ and $n\geq 3$, we have,
\begin{align*}
A_i+B_i+C_i+D_i-E_i\geq 0,
\end{align*}
for sufficient large $K$ and $\kappa_1$. 
\end{lemm}
\begin{proof}
Using  (\ref{e2.20}), we have, 
\begin{eqnarray}\label{e2.27}
&&A_i+B_i+C_i+D_i-E_i\\
&\geq& e^{\kappa_i}(K(\sigma_{n-1})_i^2-\sigma_{n-1}^{pp,qq}h_{ppi}h_{qqi})+2\dsum_{l\neq i}e^{\kappa_l}\sigma_{n-1}^{ll,ii}h_{lli}^2\nonumber\\
&&-\dfrac{1}{\log P}\dsum_{l\neq
i}e^{\kappa_l}\sigma_{n-1}^{ii}h_{lli}^2+\dfrac{1+\log
P}{P\log P}\dsum_{l\neq
i}e^{\kappa_l+\kappa_i}\sigma_{n-1}^{ii}h_{lli}^2 \nonumber\\
&&+ 2\dsum_{l\neq
i}\dfrac{e^{\kappa_l}-e^{\kappa_i}}{\kappa_l-\kappa_i}\sigma_{n-1}^{ll}h_{lli}^2
+e^{\kappa_i}\sigma_{n-1}^{ii}h_{iii}^2
-\frac{1+\log
P}{P\log P}e^{2\kappa_i}\sigma_{n-1}^{ii}h_{iii}^2 \nonumber\\
&&-2\frac{1+\log P}{P\log P}\dsum_{l\neq
i}e^{\kappa_i+\kappa_l}\sigma_{n-1}^{ii}h_{iii}h_{lli}. \nonumber
\end{eqnarray}

We claim that the following inequality holds for sufficient large $\kappa_1$,
\begin{align}\label{e2.28}
e^{\kappa_i}(K(\sigma_{n-1})_i^2-\sigma_{n-1}^{pp,qq}h_{ppi}h_{qqi})+2\dsum_{l\neq
i}\dfrac{e^{\kappa_l}-e^{\kappa_i}}{\kappa_l-\kappa_i}\sigma_{n-1}^{ll}h_{lli}^2
\geq\frac{1}{\log
P}e^{\kappa_i}\sigma_{n-1}^{ii}h_{iii}^2.
\end{align}
In view of Proposition \ref{lea}, we need to prove that, for given arbitrary small constant $\epsilon$, if $\kappa_1$ is sufficient large, we have,
\begin{align}\label{e2.29}
2\dfrac{1-e^{\kappa_l-\kappa_i}}{\kappa_i-\kappa_l}\kappa_1\geq
2\frac{n-1}{n}-\epsilon,
\end{align}
for all $l\neq i$. We divide into three cases to discuss.
\par
\noindent Case (i): $\kappa_l\geq\kappa_i$. In this case, we obviously have,
$$
\dfrac{1-e^{\kappa_l-\kappa_i}}{\kappa_i-\kappa_l}=\dfrac{e^{\kappa_l-\kappa_i}-1}{\kappa_l-\kappa_i}\geq
1.
$$
It is easy to get (\ref{e2.29}) for sufficient large $\kappa_1$. \\
\noindent Case (ii): $\kappa_i-\kappa_l\geq C_0$ where we take, 
$$C_0\geq \log\frac{2(n-1)}{n\epsilon}.$$ Then, we have, \begin{align*}
2\dfrac{1-e^{\kappa_l-\kappa_i}}{\kappa_i-\kappa_l}\kappa_1\geq
\frac{2\kappa_1}{\kappa_i-\kappa_l}(1-e^{-C_0}).
\end{align*}
Since $0<\kappa_i\leq \kappa_1$, if $\kappa_l\geq 0$, it is easy to see,
 $$\frac{2\kappa_1}{\kappa_i-\kappa_l}\geq  2\frac{n-1}{n}.$$ If $\kappa_l<0$, in $\Gamma_{n-1}$, we only have one negative eigenvalue, by Lemma \ref{le6}, we have,
$$\frac{2\kappa_1}{\kappa_i-\kappa_l}\geq 2\frac{n-1}{n}.$$ 
Combining the previous four inequalities and $n\geq 3$, we have \eqref{e2.29}.\\
\noindent Case (iii): $0<\kappa_i-\kappa_l\leq C_0$ where $C_0$ is defined in the previous case. In this case, using mean value theorem, we have, 
$$
\dfrac{1-e^{\kappa_l-\kappa_i}}{\kappa_i-\kappa_l}
=\dfrac{1}{e^{\kappa_i}}\dfrac{e^{\kappa_i}-e^{\kappa_l}}{\kappa_i-\kappa_l}
=\dfrac{e^{\xi}}{e^{\kappa_i}}\geq\dfrac{e^{\kappa_{l}}}{e^{\kappa_i}}\geq
e^{-C_{\varepsilon_{N}}}.
$$
Here $\xi$ is the mean value of $\kappa_i$ and $\kappa_l$. Since $\kappa_1$ is sufficient large, this yields (\ref{e2.29}).
In a word, \eqref{e2.28} hods for any case.

Note that, in cone $\Gamma_{n-1}$, 
\begin{align*}
2\kappa_1\sigma_{n-3}(\kappa|il)-\sigma_{n-2}(\kappa|i)
=&2\kappa_1\sigma_{n-3}(\kappa|il)-\kappa_l\sigma_{n-3}(\kappa|il)-\sigma_{n-2}(\kappa|il)\\
\geq&\kappa_1\sigma_{n-3}(\kappa|il)-\sigma_{n-2}(\kappa|il)\\
=&\kappa_1^2\sigma_{n-4}(\kappa|il1)+\kappa_1\sigma_{n-3}(\kappa|il1)-\sigma_{n-2}(\kappa|il)\\
=&\kappa_1^2\sigma_{n-4}(\kappa|il1)>0,
\end{align*}
is true for all $i,l$. It implies, 
\begin{align*}
2\kappa_1\sigma_{n-1}^{ll,ii}\geq\sigma_{n-1}^{ii}.
\end{align*}
Using the above inequality, we have, 
\begin{align}\label{e2.30}
2\dsum_{l\neq i}e^{\kappa_l}\sigma_{n-1}^{ll,ii}h_{lli}^2
-\dfrac{1}{\log P}\dsum_{l\neq
i}e^{\kappa_l}\sigma_{n-1}^{ii}h_{lli}^2\geq 0.
\end{align}
On the other hand, we have,
\begin{eqnarray}\label{e2.31}
&&\dfrac{1+\log P}{P\log P}\dsum_{l\neq
i}e^{\kappa_l+\kappa_i}\sigma_{n-1}^{ii}h_{lli}^2
-2\frac{1+\log P}{P\log P}\dsum_{l\neq
i}e^{\kappa_i+\kappa_l}\sigma_{n-1}^{ii}h_{iii}h_{lli}\\
&\geq&-\dfrac{1+\log P}{P\log P}\dsum_{l\neq
i}e^{\kappa_l+\kappa_i}\sigma_{n-1}^{ii}h_{iii}^2. \nonumber
\end{eqnarray}
Inserting (\ref{e2.28}),(\ref{e2.30}) and (\ref{e2.31}) into
(\ref{e2.27}), we obtain, 
\begin{eqnarray*}
&&A_i+B_i+C_i+D_i-E_i\\
&\geq& \frac{1}{\log
P}e^{\kappa_i}\sigma_{n-1}^{ii}h_{iii}^2+e^{\kappa_i}\sigma_{n-1}^{ii}h_{iii}^2
-\frac{1+\log P}{P\log
P}e^{2\kappa_i}\sigma_{n-1}^{ii}h_{iii}^2\nonumber\\
&&-\dfrac{1+\log P}{P\log P}\dsum_{l\neq
i}e^{\kappa_l+\kappa_i}\sigma_{n-1}^{ii}h_{iii}^2\nonumber\\
&=&0\nonumber.
\end{eqnarray*}

\end{proof}

For the negative part, we have the following estimate. 

\begin{lemm}\label{le3}
If $-\kappa_i\geq \delta \kappa_1$ and $n\geq 3$, then we also have,
\begin{align*}
A_i+B_i+C_i+D_i-E_i\geq 0,
\end{align*}
for sufficient large $K$ and $\kappa_1$. 
\end{lemm}
\begin{proof}
Firstly, for sufficient large constant $K$, by Lemma \ref{Guan}, we have $A_i> 0$. 
In this case, the only possible negative eigenvalue is $\kappa_n$. By Lemma \ref{le6}, we know that $-\kappa_i<\dfrac{1}{n-1}\kappa_1$. Then using the similar argument of the inequality \eqref{e2.29} in the previous Lemma, we have, 
$$\dfrac{5\kappa_1}{3}\dfrac{e^{\kappa_l}-e^{\kappa_i}}{\kappa_l-\kappa_i}\geq \frac{5}{6}(2\frac{n-1}{n}-\epsilon)e^{\kappa_l}.$$
Since $n\geq 3$, the coefficient of the right hand side in the above inequality is bigger than $1$ for sufficient small $\epsilon$. Hence, using Lemma \ref{leR}, we have,  
\begin{align}\label{e2.35}
\dfrac{5}{3}\dsum_{l\neq i}e^{\kappa_l}\sigma_{n-1}^{ll,ii}h_{lli}^2
+ \dfrac{5}{3}\dsum_{l\neq
i}\dfrac{e^{\kappa_l}-e^{\kappa_i}}{\kappa_l-\kappa_i}\sigma_{n-1}^{ll}h_{lli}^2
-\dfrac{1}{\log P}\dsum_{l\neq
i}e^{\kappa_l}\sigma_{n-1}^{ii}h_{lli}^2\geq 0.
\end{align}
Using \eqref{CauS}, \eqref{e2.35} and \eqref{e2.20}, we obtain,
\begin{eqnarray*}
&&B_i+C_i+D_i-E_i \\
&\geq&\frac{1+\log P}{P\log
P}e^{\kappa_1+\kappa_i}\sigma_{n-1}^{ii}h_{11i}^2+e^{\kappa_i}\sigma_{n-1}^{ii}h_{iii}^2 \nonumber\\
&&-\frac{1+\log P}{P\log P}\dsum_{l\neq
1}e^{\kappa_l+\kappa_i}\sigma_{n-1}^{ii}h_{iii}^2
-2\frac{1+\log P}{P\log P}e^{\kappa_i+\kappa_1}\sigma_{n-1}^{ii}h_{iii}h_{11i} \nonumber\\
&&+\dfrac{1}{3}e^{\kappa_1}\sigma_{n-1}^{11,ii}h_{11i}^2 +
\dfrac{1}{3}\dfrac{e^{\kappa_1}-e^{\kappa_i}}{\kappa_1-\kappa_i}\sigma_{n-1}^{11}h_{11i}^2 .   \nonumber
\end{eqnarray*}
The last expression is similar to (\ref{e2.23}). Thus, using similar argument in Lemma \ref{le2}, it is nonnegative. 
\end{proof}

Now, we are in the position to prove our main theorem. \\

\noindent {\bf Proof of Theorem \ref{theo2}:} For $n\geq 3$, using Lemma \ref{le2}, Lemma \ref{le1} and
Lemma \ref{le3} in (\ref{e2.18}), we obtain,

\begin{eqnarray*}
0&\geq&\sigma_{n-1}^{ii}\phi_{ii}\\
&\geq &\frac{1}{P\log P}\dsum_i(A_i+B_i+C_i+D_i-E_i)\nonumber\\
&&+(N-1)\sigma_{n-1}^{ii}h_{ii}^2+N\frac{\sigma_{n-1}^{ii}h_{ii}^2\langle
X,\p_i\rangle^2}{u^2}-\dfrac{C+CKh_{11}^2}{\log P}\nonumber\\
&\geq&(N-1)c_0h_{11}-\dfrac{C+CKh_{11}^2}{\log
P}\nonumber.
\end{eqnarray*}
Here we have used $$\sigma_{n-1}^{11}h_{11}\geq c_0.$$
Choosing sufficient large $N$, we get an upper bound of $h_ {11}$. 

For $n=2$, the equation is a quasi linear elliptic equation. The $C^2$ estimate is well known.

\section{Some application}
Let's gives some applications. The first is to prove existence result, Theorem \ref{exist}. \\

{\bf Proof of Theorem \ref{exist}: }  
We use continuity method to solve the existence result.  For $0\leq t\leq 1$, according to \cite{CNS5}, we consider the family of functions,
\begin{eqnarray}
f^t(X,\nu)=tf(X,\nu)+(1-t)C_n^2[\frac{1}{|X|^k}+\varepsilon(\frac{1}{|X|^k}-1)],\nonumber
\end{eqnarray}
where $\varepsilon$ is sufficient small constant satisfying $$0< f_0\leq \min_{r_1\leq \rho\leq r_2} (\frac{1}{\rho^k}+\varepsilon(\frac{1}{\rho^k}-1)),$$ and $f_0$ is some positive constant. The $C^0$ and $C^1$ estimates is same to the proof in \cite{GRW}. 
For $n\geq 3$, the $C^2$ estimate comes from Theorem \ref{theo2}. The openness comes from \cite{CNS5}. By continuity method and Evans-Krylov theory, we obtain Theorem \ref{exist}.  We complete our proof. \\

The proof of the Corollary \ref{Coro}  is similar to Theorem \ref{theo2}. Using the Corollary and the boundary estimates obtained in \cite{B}, we have the following existence result for Dirichelt problem.

\begin{theo}
Suppose $\Omega\subset \mathbb R^n$ is a bounded domain with smooth boundary. Suppose $f(p, u, x)\in C^2(\mathbb R^n\times \mathbb R\times \bar\Omega)$ is a positive function with $f_u\ge 0$. Suppose there is a subsolution $\underline{u}\in C^3(\bar \Omega)$ satisfying  \begin{equation}\label{2.1.1}
\left\{\begin{matrix}\sigma_{n-1}[D^2\underline{u}]&\geq &f(x,\underline{u},D\underline{u}), \\
\underline{u}|_{\partial \Omega}&=&\varphi.\end{matrix}\right.
\end{equation}
 then the Dirichlet problem (\ref{1.3}) has a unique $C^{3,\alpha}$ solution $u$ for any $0<\alpha<1$. \end{theo}

Then, we consider the prescribed curvature problem for spacelike graph hypersurface in Minkowski space.

 We present some setting of that problem. If function $u$ is the description function and  hypersurface $M=\text{ graph } u$. $u$ is defined in some bounded domain $\Omega\subset \mathbb{R}^n$.  The Minkowski space $\mathbb{R}^{n,1}$ is defined by the following metric,
$$ds^2=dx_1^2+\cdots dx_{n}^2-dx_{n+1}^2.$$ Since $M$ is space like, in \cite{Ba}, the uniformly $C^1$ bound has been obtained for equation \eqref{1.66}. Namely, there is some constant $\theta$, such that, $$\sup_{\bar{\Omega}}|D u|\leq \theta <1.$$ The induce metric on $M$ is,
$$g_{ij}=\delta_{ij}-D_iuD_ju, \ \  1\leq i,j\leq n.$$ The second fundamental form is,
$$h_{ij}=\frac{D_{ij}u}{\sqrt{1-|Du|^2}}.$$ We still denote the principal curvature of $M$ by $\kappa_1,\cdots,\kappa_n$. We also define the second fundamental form,  
\begin{equation}
h_{ij}=\langle\partial_iX,\partial_j\nu\rangle,
\end{equation}
Here $\langle,\rangle$ is the Minkowski inner product defined by metric $ds^2$ in the above. 
Then, for space like hypersurface, we have different Gauss formula and Gauss equation,
\begin{equation}
\begin{array}{rll}
X_{ij}=& h_{ij}\nu\quad {\rm (Gauss\ formula)}\\
R_{ijkl}=&-(h_{ik}h_{jl}-h_{il}h_{jk})\quad {\rm (Gauss\ equation)},\\
\end{array}
\end{equation}
where $R_{ijkl}$ is the $(4,0)$-Riemannian curvature tensor. Hence, the communication formula also
change a little bit, 
\begin{equation}
\begin{array}{rll}
h_{ijkl}=& h_{ijlk}+h_{mj}R_{imlk}+h_{im}R_{jmlk}\\
=& h_{klij}-(h_{mj}h_{il}-h_{ml}h_{ij})h_{mk}-(h_{mj}h_{kl}-h_{ml}h_{kj})h_{mi}.\\
\end{array}
\end{equation}
Now let's give the proof of Theorem \ref{theo5}.\\

\noindent {\bf Proof of Theorem \ref{theo5}:}
$C^0$ estimate comes from comparison principal. We also have the $C^1$ estimate. For $C^2$ estimates on the boundary, using the sub solution and the $C^2$ boundary estimate argument  \cite{B}, we can obtain it. For the interior, we use the similar trick in section 3. Hence, for function $u$, we consider the following test function,
$$
\phi=\log\log P+\dfrac{N}{2}|Du|^2.
$$
where function $P$ is also defined by $$P=\dsum_le^{\kappa_l}. $$ Suppose that $M$ achieve its maximum
value in $\Omega$ at some point $x_0$. We can assume that matrix
$(u_{ij})$ is diagonal by rotating the coordinate, and $\kappa_1\geq
\kappa_2\geq\cdots\geq \kappa_n$. 
Hence, at $x_0$, differentiating $\phi$ twice, we have
\begin{equation}\label{e3}
\phi_i=\dfrac{P_i}{P\log P}+N u_iu_{ii}=0,
\end{equation}
and,
\begin{equation}\label{4}
\phi_{ii}=\dfrac{P_{ii}}{P\log P}-\dfrac{(1+\log P)P_{i}^2}{(P\log
P)^2}+\dsum_sN u_su_{sii}+N u_{ii}^2.
\end{equation}
Similar to the calculation \eqref{e2.11} and \eqref{e2.13}, we have, 
\begin{eqnarray}\label{e5}
&&\sigma_{n-1}^{ii}\phi_{ii}\\
&=&\frac{1}{P\log
P}[\sum_le^{\kappa_l}\sigma_{n-1}^{ii}h_{ii,ll}-(n-1)f\sum_le^{\kappa_l}h_{ll}^2+\sigma_{n-1}^{ii}h_{ii}^2\sum_le^{\kappa_l}h_{ll}
\nonumber\\
&&+\sum_l\sigma_{n-1}^{ii}e^{\kappa_l}h_{lli}^2+\sum_{\alpha\neq \beta}\sigma_{n-1}^{ii}\frac{e^{\kappa_{\alpha}}-e^{\kappa_{\beta}}}{\kappa_{\alpha}-\kappa_{\beta}}h_{\alpha\beta i}^2-(\frac{1}{P}+\frac{1}{P\log P})\sigma_{n-1}^{ii}P_i^2]\nonumber\\
&&+\dsum_sN
u_s\sigma_{n-1}^{ii}u_{sii}+\sigma_{n-1}^{ii}N u_{ii}^2. \nonumber
\end{eqnarray}

\par
At $x_0$, differentiating  equation (\ref{1.66}) twice, we have,
\begin{equation}\label{e6}
\sigma_{n-1}^{ii}h_{iij}=f_j+f_uu_j+f_{p_j}u_{jj},
\end{equation}
and
\begin{align}
\sigma_{n-1}^{ii}h_{iijj}+\sigma_{n-1}^{pq,rs}h_{pqj}h_{rsj} \geq
-C-Cu_{jj}^2+\dsum_sf_{p_s}u_{sjj}.      \label{7}
\end{align}
Inserting  (\ref{7}) into (\ref{e5}),
 we have
\begin{eqnarray}
&&\sigma_{n-1}^{ii}\phi_{ii}\label{e18}\\
&\geq &\frac{1}{P\log P}[\sum_le^{\kappa_l}(K(\sigma_{n-1})_l^2-\dsum_{p\neq q}\sigma_{n-1}^{pp,qq}h_{ppl}h_{qql}+\dsum_{p\neq q}\sigma_{n-1}^{pp,qq}h_{pql}^2)\nonumber\\
&&+\sum_l\sigma_{n-1}^{ii}e^{\kappa_l}h_{lli}^2+\sum_{\alpha\neq
\beta}\sigma_{n-1}^{ii}\frac{e^{\kappa_{\alpha}}-e^{\kappa_{\beta}}}{\kappa_{\alpha}-\kappa_{\beta}}h_{\alpha\beta
i}^2-\frac{1+\log P}{P\log
P}\sigma_{n-1}^{ii}P_i^2\nonumber\\
&&-CP-CKPh_{11}^2]+N\sigma_{n-1}^{ii}u_{ii}^2\nonumber\\
&\geq& \frac{1}{P\log P}\dsum_i(A_i+B_i+C_i+D_i-E_i)+N\sigma_{n-1}^{ii}h_{ii}^2(1-|Du|^2)\nonumber\\
&&-\frac{C+CK \kappa_1^2}{\log P}\nonumber.
\end{eqnarray}
Here, the definition of $A_i,B_i,C_i,D_i,E_i$ is same meaning as the previous section. Thus, since $\theta$ is a constant smaller than $1$, we obtain the uniformly bound of $h_{11}$.  The openness is standard. Using the continuity method and Evans-Krylov theory, we obtain our theorem.

\bigskip

\noindent {\it Acknowledgement:} The authors wish to thank  Professor Pengfei Guan for his valuable suggestions and comments. They also
 thank  the Shanghai Centre for Mathematical Sciences for their partial support.  The first author would  like to thank Fudan University for their support and hospitality.


\begin{thebibliography}{99}
%\bibitem{A} Alexandrov, A.D.,{ Zur Theorie der gemischten Volumina von konvexen Kˆrpern, III. 
%Die Erweiterung zweier Lehrs‰tze Minkowskis ¸ber die konvexen Polyeder auf beliebige konvexe Fl‰chen}, (in Russian). \textit{Mat. Sb.}, \textbf{3}, (1938), 27-46.

\bibitem{A1} A.D. Alexandrov,
{\em  Existence and uniqueness of a convex
surface with a given integral curvature}.{ Doklady Acad. Nauk Kasah SSSR},
{36} (1942), 131-134.

\bibitem{A2}
 A.D. Alexandrov,{\em Uniqueness theorems for surfaces in the
large. I} (Russian).{ Vestnik Leningrad. Univ.,} {11},
(1956), 5-17. English translation: AMS Translations, series 2,
{21}, (1962), 341-354.

\bibitem {BK}
I. Bakelman, and B. Kantor,
{\em Existence of spherically homeomorphic hypersurfaces in Euclidean
space with prescribed mean curvature}.
{ Geometry and Topology,} Leningrad,
{1}, (1974), 3-10.

\bibitem{Ball} J. Ball,{\em Differentiability properties of symmetric and isotropic functions}. {Duke Math. J.}, {51}, (1984),  699-728.

\bibitem{BS} R. Bartnik and L. Simon, {\em Spacelike hypersurfaces with prescribed boundary values and mean curvture }. Comm. Math. Pays. 87,131-152(1982).

\bibitem{Ba0} P. Baryard, {\em Probl\'eme de Dirichlet pour la courbure scalaire dans $\mathbb{R}^{3,1}$}, C.R. Acad. Sci. Paris S\'er. I Math. 332, 219-222(2001).
\bibitem{Ba} P. Bayard, {\em Dirichlet problem for space-like hypersurfaces with prescribed scalar curvature in $\mathbb{R}^{n,1}$}, Calc. Var. 18, 1-30(2003). 


\bibitem{CNS1} L. A. Caffarelli, L. Nirenberg and J. Spruck,
{\em The Dirichlet problem for nonlinear  second order elliptic equations
I. Monge-Amp\`ere equations}.
{Comm. Pure Appl. Math.}, {37}, (1984), 369-402.

\bibitem{CNS3}
L. A. Caffarelli, L. Nirenberg and J. Spruck,
{\em The Dirichlet problem for nonlinear second order elliptic equations,  III:
Functions of the eigenvalues of the Hessian}.
{Acta Math.}, {155}, (1985),261 - 301.

\bibitem{CNS5}
L. A. Caffarelli, L. Nirenberg and J. Spruck,
{\em Nonlinear second order elliptic equations IV:
 Starshaped compact Weigarten hypersurfaces}.
{ Current topics in partial differential equations,} Y.Ohya, K.Kasahara
and N.Shimakura (eds),
 Kinokunize, Tokyo, 1985,
 1-26.

\bibitem{CNSV} L. A. Caffarelli, L. Nirenberg, J. Spruck, {\em Nonlinear second order elliptic equations. V. The Dirichlet problem for Weingarten hypersurfaces}, Comm.  Pure and Appl. Math. 41 (1988), pp. 41-70.
\bibitem{CY}
S.Y. Cheng and S.T.  Yau, {\em On the regularity of the solution
of the $n$-dimensional Minkowski problem}.{ Comm. Pure Appl. Math.,}
{29}, (1976), 495-516.

\bibitem{CW} K.S. Chou and X.J. Wang, {\em A variational theory of the Hessian equation}. {Comm. Pure Appl. Math.}, {54}, (2001), 1029-1064.

\bibitem{De} P. Delano\"e, {\em The Dirichlet problem for an equation of given Lorentz-Gaussian curvature}. (Russian) Ukrain. Mat. Zh. 42, 1704-1710(1990). Eng. trails.: Ukrain. Math. J. 42, 1538-1545.   

\bibitem{Ger1} C. Gerhardt, {\em Hypersurfaces of prescribed curvature in Lorentzian manifolds}. Indiana Univ. Math. J. 49,1125-1153 (2000).

\bibitem{Ger2} C. Gerhardt, { Hypersurfaces of prescribed scalar curvature in Lorentzian manifolds}. \textit{J. Reine Angew. Math.}, \textbf{554}, (2003), 157-199.




\bibitem{B} B. Guan, {\em  The Dirichlet problem for Hessian equations on
Riemannian manifolds}. {Calc. Var.}, {8}, (1999), 45-69.

\bibitem{gg} B. Guan and P. Guan,  {\em  Convex Hypersurfaces of Prescribed Curvature}.
\textit{Ann. of Math.,} {156},
(2002), 655-674.


\bibitem{GLL} P. Guan, J. Li  and Y.Y. Li,  {\em Hypersurfaces of Prescribed Curvature Measure}. {Duke Math. J.}, {161}, (2012), 1927-1942.

\bibitem{GLM} P.Guan, C.S. Lin and X. Ma, {\em The Existence of Convex Body with Prescribed Curvature Measures}.
{Int. Math. Res. Not.,} (2009) 1947-1975.

\bibitem{GRW} P. Guan, C. Ren and Z. Wang, {\em Global $C^2$ estimates
for curvature equation of convex solution.} Comm. Pure Appl. Math. LXVIII(2015) 1287-1325. 
%\bibitem{GLM1}
%Guan, P., Lin, C-S and  Ma, X., { The Christoffel-Minkowski problem II: Weingarten curvature equations}.
%\textit{Chin. Ann. Math., }\textbf{27B}, (2006), 595-614.

%\bibitem{GM}
%Guan, P. and Ma, X.,{ Christoffel-Minkowski problem I: convexity of solutions of a Hessian equation}.\textit{  Invent. Math.,} \textbf{151},
%(2003), 553-577.

%\bibitem{HS} Huisken, G. and Sinestrari, C., Convexity estimates for mean curvature flow and singularities of mean convex surfaces.
%\textit{Acta. Math.,} \textbf{183}, (1999), 45-70.

\bibitem{I1}
N. Ivochkina,{ \em Solution of the Dirichlet problem for curvature
equations of order m}, { Mathematics of the USSR-Sbornik,} {67}, (1990),
 317-339.

\bibitem{I}
N.  Ivochkina,  {\em  The Dirichlet problem for the equations of curvature of
order m.} { Leningrad Math. J.} {2-3}, (1991), 192-217.

\bibitem{LRW} M. Li, C. Ren and Z. Wang, {\em An interior estimate for convex solutions and a rigidity theorem}, to appear in J. Funct. Anal. 


\bibitem{N}
L. Nirenberg, {\em  The Weyl and Minkowski problems in differential
geometry in the large}. {Comm. Pure Appl. Math.,} {6}, (1953),
337-394.

\bibitem{P1}
A.V. Pogorelov, {\em On the question of the existence of a convex
surface with a given sum principal radii of curvature ( in
Russian)}. { Uspekhi Mat. Nauk.}, {8}, (1953), 127-130.

\bibitem{P3}  A.V. Pogorelov, {\em The Minkowski Multidimensional
Problem}.  John Wiley, 1978.


\bibitem{Sch} O.C. Schn\"urer,{\em The Dirichlet problem for Weingarten hypersurfaces in Lorentzian manifolds}. Math. Z. 242, 159-181(2002). 

\bibitem{SX} J. Spruck and L. Xiao, {\em A note on starshaped compact hypersurfaces with prescribed scalar curvature in space form}. arXiv: 1505.01578.
\bibitem{TW}
A.  Treibergs and S.W. Wei, 
{\em  Embedded hypersurfaces with prescribed mean curvature}.
{ J. Diff. Geom. ,}
{18}, (1983),  513-521.









\bibitem{U} J. Urbas, {\em The Dirichlet problem for the equation of prescribed scalar curvature in Minkowski space}, Calc. Var. 18, 307-316 (2003). 
\end{thebibliography}
\end{document}